\PassOptionsToPackage{style=alphabetic,backend=biber,maxbibnames=99, maxcitenames=99, url=false}{biblatex} 
\PassOptionsToPackage
{left=2.5cm,
	right=2.5cm,
	top=1cm,
	bottom=1cm,
	footskip=18pt,
	includehead,
	includefoot,
	headsep=12pt}
{geometry} 
\documentclass[reqno,11pt]{article}
\usepackage{amsthm}
\usepackage{amsfonts}
\usepackage[colorlinks=true,linkcolor=blue,citecolor=blue]{hyperref}
\usepackage{biblatex}
\usepackage{geometry} 
\usepackage{stmaryrd}
\usepackage{longtable}
\usepackage{amssymb}
\usepackage{amsmath}
\usepackage{graphicx}
\usepackage{eucal, bbm}
\usepackage{enumitem}
\usepackage{dsfont}
\usepackage{longtable}
\usepackage{array}

\numberwithin{equation}{section}
\theoremstyle{plain}

\newtheorem{theorem}{Theorem}[section]
\newtheorem{corollary}[theorem]{Corollary}
\newtheorem{problem}[theorem]{Problem}
\newtheorem{lemma}[theorem]{Lemma}
\newtheorem{proposition}[theorem]{Proposition}

\newtheorem{assumption}[theorem]{Assumption}
\newtheorem{definition}[theorem]{Definition}
\newtheorem{remark}[theorem]{Remark}

\theoremstyle{remark}
\newtheorem{example}[theorem]{Example}

\renewcommand{\d}{{\mathrm d}}

\newcommand{\restr}[1]{\lower3pt\hbox{$|_{#1}$}}

\newcommand{\mres}{\mathbin{\vrule height 1.6ex depth 0pt width 0.13ex\vrule height 0.13ex depth 0pt width 1.3ex}}

\newcommand{\nchi}{{\raise.3ex\hbox{$\chi$}}}

\newcommand{\B}{\mathbb{B}}

\newcommand{\bbM}{\mathbb{M}}
\newcommand{\N}{\mathbb{N}}

\renewcommand{\P}{\mathbb{P}}

\newcommand{\R}{\mathbb{R}}

\newcommand{\W}{\mathbb{W}}
\newcommand{\X}{\mathbb{X}}
\newcommand{\Y}{\mathbb{Y}}
\newcommand{\Z}{\mathbb{Z}}

\newcommand{\XX}{\mathscr{X}}

\newcommand{\cA}{{\ensuremath{\mathcal A}}}
\newcommand{\cB}{{\ensuremath{\mathcal B}}}
\newcommand{\calC}{{\ensuremath{\mathcal C}}}

\newcommand{\cF}{{\ensuremath{\mathcal F}}}

\newcommand{\cH}{{\ensuremath{\mathcal H}}}

\newcommand{\cP}{{\ensuremath{\mathcal P}}}

\newcommand{\cT}{{\ensuremath{\mathcal T}}}

\newcommand{\mmu}{{\boldsymbol \mu}}

\newcommand{\sfg}{{\mathsf g}}

\newcommand{\sfw}{{\mathsf W}}

\newcommand{\sfC}{{\mathsf C}}

\newcommand{\sfW}{{\mathcal W}}

\newcommand{\frg}{{\mathfrak  g}}

\newcommand{\frP}{{\mathfrak  P}}

\newcommand{\rmd}{{\mathrm d}}

\newcommand{\rmt}{{\mathrm t}}

\newcommand{\rmA}{{\mathrm A}}
\newcommand{\rmB}{{\mathrm B}}

\newcommand{\rmF}{{\mathrm F}}

\newcommand{\rmH}{{\mathrm H}}

\newcommand{\rmN}{{\mathrm N}}

\newcommand{\rmQ}{{\mathrm Q}}

\newcommand{\rmT}{{\mathrm T}}

\newcommand{\tta}{{\mathtt a}}

\newcommand{\ttc}{{\mathtt c}}

\newcommand{\ttm}{{\mathtt m}}

\newcommand{\ttt}{{\mathtt t}}

\newcommand{\ttA}{{\mathtt A}}

\newcommand{\ttG}{{\mathtt G}}

\newcommand{\ttM}{{\mathtt M}}
\newcommand{\ttN}{{\mathtt N}}

\newcommand{\ttP}{{\mathtt P}}

\renewcommand{\P}{\ensuremath{\mathbb{M}}}

\title{Optimal transport between laws of random probability measures and the strict Monge problem
}
\begin{document}
	

	\author{
		Alessandro Pinzi\thanks{Bocconi University, Department of Decision Sciences, via Roentgen 1, 20136 Milano (Italy). Email:\textsf{alessandro.pinzi@phd.unibocconi.it}. Orcid: \textsf{https://orcid.org/0009-0007-9146-5434}}
	}
	
	
	\maketitle

	\begin{abstract}
		We consider an optimal transport problem between laws of random probability measures: given a base cost function,
		we build the associated OT cost between probability measures that in turn we
		use to define the OT cost between probability measures over probability measures. This setting admits a finer reformulation in terms of laws of random couplings, which retain more information than ordinary couplings. One of the main contributions of the paper is the characterization of the optimal ones in terms of Kantorovich potentials.
		
		Similarly, we also introduce the strict Monge problem, whose admissible competitors are more restrictive than in the usual Monge formulation. In this setting, we will
		give sufficient conditions under which the value of this problem is the same
		as the one considered above, in the spirit of the result by A. Pratelli. Then, for $p>1$,
		when the underlying cost is the distance to the power $p$ in a strictly convex Banach space, we will
		give sufficient conditions under which the optimal random coupling is unique
		and induced by a solution of the strict Monge problem, resembling the Brenier
		theorem.
	\end{abstract}
	
	{\small
		
		
	}
	
	{\small\tableofcontents}
	
	\newcommand{\mduality}[2]{[#1,#2]}
	\newcommand{\rW}[1]{\cP_{#1}(\rmH)}
	\newcommand{\cPP}{\cP\kern-4.5pt\cP}
	\newcommand{\RW}[1]{\cP\kern-4.5pt\cP_{#1}(\rmH)}
	\newcommand{\RWfin}[1]{\ensuremath{\cP\kern-4.5pt\cP_{#1}(\R^d)}}
	\newcommand{\RWr}[1]{\cP\kern-4.5pt\cP_{#1}^{r\kern-1pt r}(\rmH)}
	\newcommand{\RWrd}[1]{\ensuremath{\cP\kern-4.5pt\cP_{#1}^{r\kern-1pt r}(\R^d)}}
	\newcommand{\RWru}[1]{\ensuremath{\cP\kern-4.5pt\cP_{#1}^{r\kern-1pt r}(\R)}}
	\newcommand{\RWgr}[1]{\cP\kern-4.5pt\cP_{#1}^{gr\kern-1pt r}(\rmH)}
	\newcommand{\RWgrd}[1]{\ensuremath{\cP\kern-4.5pt\cP_{#1}^{g\kern-1pt r\kern-1pt r}(\R^d)}}
	\newcommand{\RWW}[1]{\cP\kern-4.5pt\cP_{#1}(\rmH\times\rmH)}
	\newcommand{\RWWdet}[1]{\cP\kern-4.5pt\cP_{#1}^{\rm det} (\rmH\times\rmH)}
	\newcommand{\msc}[2]{[#1,#2]}
	\newcommand{\mmsc}[2]{[\kern-2pt[#1,#2]\kern-2pt]}
	\newcommand{\shto}{\shortrightarrow}
	\newcommand{\Rinf}{\ensuremath{\R\cup\{\infty\}}}
	\newcommand{\bpartial}{\boldsymbol{\partial}}
	\newcommand{\bpartialt}{\boldsymbol{\partial}_{\rmt}}
	\newcommand{\Ast}{{\mathop{\scalebox{1}{\raisebox{-0.2ex}{$\ast$}}}}}%
	\newcommand{\RGamma}{\mathrm{R}\Gamma}
	\newcommand{\lift}[1]{#1^\ell}
	\newcommand{\cPdet}[1]{{\cP^{\rm det}_{#1}}}
	\newcommand{\LF}{\hat{\brmF}}
	\newcommand{\bttP}{{\boldsymbol{\mathtt P}}}

	\renewcommand{\X}{\mathrm{X}}
	\renewcommand{\Y}{\mathrm{Y}}
	\renewcommand{\B}{\mathrm{B}}
	\renewcommand{\W}{W}
	\renewcommand{\Z}{Z}
	\renewcommand{\XX}{\mathcal{X}}
	
	\newcommand{\C}{\mathsf{C}}
	\newcommand{\M}{\ttM}
	\newcommand{\m}{\ttm}
	\newcommand{\n}{\ttm}
	\newcommand{\PP}{\mathcal{P}}
	\renewcommand{\mmu}{\mu}

	\newcommand{\cost}{\mathop{\mathsf{c}}\nolimits}
	
	\newcommand{\Law}{\operatorname{Law}}
	
	\newcommand{\PPdiff}{\PP^{\operatorname{diff}}}
	
	\newcommand{\PPdet}{\PP^{\operatorname{det}}}

	\section{Introduction}
	In this paper, we study an Optimal Transport problem between laws of random probability measures in an iterated fashion: given a cost function $\cost\colon\X_1\times \X_2 \to [0,+\infty]$, we associate to it the optimal transport cost $\C\colon\PP(\X_1)\times \PP(\X_2) \to [0,+\infty]$ defined as 
	\begin{equation}
		\C(\mu_1,\mu_2) := \min \left\{ \int_{\X_1\times \X_2} \cost(x_1,x_2) d\pi(x_1,x_2) \ : \ \pi \in \Gamma(\mu_1,\mu_2) \right\},
	\end{equation}
	and we use it again to define an OT cost $\calC\colon \PP(\PP(\X_1)) \times \PP(\PP(\X_2)) \to [0,+\infty]$ as
	\begin{equation}\label{eq: OTRM intro}
		\calC(\ttM_1,\ttM_2):=\min \left\{ \int_{\PP(\X_1)\times \PP(\X_2)} \C(\mu_1,\mu_2) d\Pi(\mu_1,\mu_2) \ : \ \Pi \in \Gamma(\M_1,\M_2) \right\}.
	\end{equation}    
	We will mainly work in the setting in which $\X_1$ and $\X_2$ are Polish spaces and $\cost$ is lower semicontinuous, sometimes treating separately the case in which $\cost$ is continuous. With $\Gamma(\mu_1,\mu_2)$ we denote the set of transport couplings, see also Section \ref{subsec: C-conc theory}. Note that we will always write that the minimum is attained even in the case in which the value is $+\infty$, taking care in saying when it is finitely valued.
	
	Clearly, the general theory already gives a characterization of the optimal transport plans in terms of the optimal dual potentials. However, in this setting \eqref{eq: OTRM intro} can be reformulated as
	\begin{equation}\label{eq: OTRM rand intro}
		\calC(\M_1,\M_2) = \min\left\{  \int_{\PP(\X_1\times \X_2)} \int_{\X_1\times \X_2} \cost(x_1,x_2) d\pi(x_1,x_2) d\ttP(\pi) \ : \ \ttP \in \RGamma  (\M_1,\M_2) \right\},
	\end{equation}
	where $\RGamma(\ttM_1,\ttM_2)$ is the set of \textit{(laws of) random couplings}, that are measures $\ttP\in \PP(\PP(\X_1\times \X_2))$ such that the double push-forward of the projections $\operatorname{pr}^i:\X_1\times \X_2 \to \X_i$, for $i=1,2$, returns the given random measures:
	\[\operatorname{pr}^1_{\sharp \sharp}\ttP = \ttM_1, \quad \operatorname{pr}^2_{\sharp\sharp} \ttP = \ttM_2.\]
	
	\begin{example}\label{ex: wass}
		A natural example of this construction is the Wasserstein-on-Wasserstein distance: let $(\X,\rmd)$ be a complete and separable metric space. In the previous setting, consider $\X_1 = \X_2 = \X$ and $\cost(x_1,x_2) = \rmd^p(x_1,x_2)$ for a fixed $p\in[1,+\infty)$. Then $\C(\mu_1,\mu_2) = \sfw_{p,\rmd}^p(\mu_1,\mu_2)$ is the $p$-th power of the $L^p$-Wasserstein distance over $\PP(\X)$, in the extended sense since we are not yet restricting to measures with finite $p$-moment and the value can be $+\infty$. Then, $\calC(\ttM_1,\ttM_2) = \sfW_{p,d}^p(\ttM_1,\ttM_2)$ is the $p$-th power of the (extended) Wasserstein-on-Wasserstein distance, called also $L^p$-Wasserstein distance on random measures. See also \ref{subsec: wass}.
	\end{example}
	
	In general, random couplings contain more information than couplings: by a measurable selection argument, we can show that the map 
	\begin{equation}
		(\operatorname{pr}^1_\sharp, \operatorname{pr}^2_\sharp)_\sharp \colon \RGamma(\ttM_1,\ttM_2) \to \Gamma(\ttM_1,\ttM_2)
	\end{equation}
	is surjective, but in general is not bijective. For example, if $\ttM_1 = \delta_{\mu_1}$ and $\ttM_2=\delta_{\mu_2}$ for some $\mu_1\in \PP(\X_1)$ and $\mu_2 \in \PP(\X_2)$, then $\Gamma(\ttM_1,\ttM_2)$ is a singleton, but $\RGamma(\ttM_1,\ttM_2)$ contains $\delta_\pi$ for all $\pi\in \Gamma(\mu_1,\mu_2)$. For more details, see Lemma \ref{lemma: rand coupl}. Thus, in this setting, understanding the structure of optimal random couplings will lead to a better understanding of the problem, rather than simply applying the general theory to \eqref{eq: OTRM intro}.
	
	An example of random couplings are the \textit{fully deterministic} ones, that are induced by maps of the form $\rmT:\X_1\times \PP(\X_1) \to \X_2$. Indeed, if $\rmT$ satisfies $\cT_\sharp \ttM_1 = \ttM_2$, with $\cT(\mu_1):= \rmT(\cdot, \mu_1)_\sharp \mu_1$, then 
	\begin{equation}\label{eq: repr fully intro}
		\ttP:= \int_{\PP(\X_1)} \delta_{(\operatorname{id},\rmT(\cdot,\mu_1))_\sharp\mu_1} d\ttM_1(\mu_1) \in \RGamma(\ttM_1,\ttM_2).
	\end{equation}
	
	Our first goal is to characterize optimal random couplings through a suitable extension of $\C$-cyclical monotonicity and $\C$-concavity. Our second goal is to introduce the \textit{strict Monge problem} using fully deterministic competitors, and to understand when it is exact and when optimal random couplings are induced by measurable maps.
	
	\bigskip
	
	\noindent\textbf{Random couplings and Kantorovich potentials.} A crucial turning point in OT theory is \textit{duality}. In the setting of \eqref{eq: OTRM intro}, the standard theory tells us that 
	\[
	\calC(\ttM_1,\ttM_2) = \sup \left\{ \int_{\PP(\X_1)} \phi_1 \, d\ttM_1 + \int_{\PP(\X_2)} \phi_2 \, d\ttM_2 \;:\; \phi_1(\mu_1)+\phi_2(\mu_2)\le \sfC(\mu_1,\mu_2) \right\}.
	\]
	From now on, we impose the following assumption: for some Borel functions $\tta_i\colon \X_i \to [0,+\infty]$
	\begin{equation}\label{eq: as intro}
		\cost(x_1,x_2)\leq \tta_1(x_1) + \tta_2(x_2) \quad \text{ and } \int_{\PP(\X_i)}\int_{\X_i} \tta_i(x_i) d\mu_i(x_i) d\ttM_i(\mu_i)<+\infty \quad \text{for }i=1,2.
	\end{equation}
	Define also $\PP_{\tta_i}(\X_i) := \{\mu_i \in \PP(\X_i) : \tta_i \in L^1(\X_i,\mu_i)\}$.
	Then, $\sfC(\mu_1,\mu_2)<+\infty$ for all $(\mu_1,\mu_2)\in \PP_{\tta_1}(\X_1) \times \PP_{\tta_2}(\X_2)$ and $\calC(\ttM_1,\ttM_2)<+\infty$, so that we can apply the general theory for optimal transport: the maximum is attained by a pair of functions $(\phi_1,\phi_2)\in L^1(\ttM_1)\times L^1(\ttM_2)$ satisfying that $\phi_1$ is $\sfC$-concave and $\phi_2$ is the $\sfC$-transform of $\phi_1$. 
	
	\begin{example}\label{ex: wass 2}
		In the Wasserstein case of Example \ref{ex: wass}, a natural choice is $\tta_1(x) = \tta_2(x) := 2^{p-1}\rmd^p(x,x_0)$ for some fixed point $x_0\in \X$. This choice gives $\PP_{\tta_1}(\X) = \PP_{\tta_2}(\X) = \PP_p(\X)$, and \eqref{eq: as intro} is equivalent to $\ttM_i \in \PP_p(\PP_p(\X))$. Moreover, in this particular case we have that $(\PP_p(\X),\sfw_{p,\rmd})$ and $(\PP_p(\PP_p(\X)),\sfW_{p,\rmd})$ are complete and separable metric spaces (see \ref{subsec: wass}). 
	\end{example}
	
	At this point, recall that, for couplings $\Pi \in \Gamma(\ttM_1,\ttM_2)$, optimality is characterized by looking closer at the theory of $\C$-concave functions. Indeed, the following are equivalent:
	\begin{enumerate}
		\item $\Pi \in \Gamma_{o,\C}(\ttM_1,\ttM_2)$;
		\item $\Pi$ is concentrated on a $\C$-cyclically monotone subset of $\PP_{\tta_1}(\X_1) \times \PP_{\tta_2}(\X_2)$ (see Definition \ref{def:C-conc});
		\item there exists $\phi_1:\PP_{\tta_1}(\X_1) \to [-\infty,+\infty)$ such that $\Pi$ is concentrated on $\partial_{\C}^+\phi_1$ (see Definition \ref{def:C-conc}).
	\end{enumerate}
	In particular, the equivalence between 2. and 3. is given by the celebrated Rockafellar--Ruschendorf theorem, that states that every $\C$-cyclically monotone subset is contained in the $\C$-superdifferential of some $\C$-concave function (see Theorem \ref{thm:RR}).
	
	One of the main results of this paper (see Theorem \ref{thm: char opt RGamma}) is a similar characterization for optimal random couplings. We will introduce the notion of \textit{total} $\sfC$\textit{-cyclically monotone} subsets $\rmF \subset \PP_{\tta_1\oplus \tta_2}(\X_1\times \X_2)$ (see Definition \ref{def: total cycl mon}) and of \textit{total} $\sfC$\textit{-superdifferential} $\partial_{\ttt,\sfC}^+\phi_1 \subset \PP_{\tta_1\oplus\tta_2}(\X_1\times \X_2)$ of a function $\phi_1\colon\PP_{\tta_1}(\X_1) \to[-\infty,+\infty)$ (see Definition \ref{def: total super}), and we will prove that, given $\ttP \in \RGamma(\ttM_1,\ttM_2)$, the following are equivalent:
	\begin{enumerate}
		\item $\ttP$ is optimal;
		\item $\ttP$ is concentrated on a total $\C$-cyclically monotone set;
		\item there exists $\phi_1:\PP_{\tta_1}(\X_1) \to [-\infty,+\infty)$ $\C$-concave such that $\ttP$ is concentrated on $\partial^+_{\ttt,\C}\phi_1$. 
	\end{enumerate}
	Moreover, we also have the `total' counterpart of Rockafellar--Ruschendorf theorem, showing that any total $\C$-cyclically monotone subset is contained on the total $\C$-superdifferential of some $\C$-concave function (see Proposition \ref{prop: char tot mon}).

	This result extends \cite[Section 4 \& 5]{PS25convex} (see also \cite{beiglbock2025brenier}), where the case $\X_1 = \X_2 = \rmH$ separable Hilbert space, with $\cost(x_1,x_2) = |x_1-x_2|^2$, was treated. However, in this specific setting, the author with G. Savar\'e proved the result heavily relying on the Lions' lifting $L^2(\Omega,\mathcal{F},\mathbb{\P};\rmH)$ of $\PP_2(\rmH)$, while in Section \ref{sec: OT gen cost} we propose a novel strategy of proof that is intrinsic and does not rely on any kind of representations. 
	
	Anyway, in Section \ref{sec: Lions}, we show that a lifting procedure can be performed in this setting as well. This is interesting on its own, and gives also a different proof of the results in Section \ref{sec: OT gen cost}, that is indeed similar to the one proposed in \cite{PS25convex}.
	
	\bigskip

	\noindent \textbf{Lions' lifting.} Assume that $\cost:\X_1\times \X_2 \to[0,+\infty]$ is continuous. Following \cite{CSS2,PS25convex}, let $(\rmQ,\cF_{\rmQ}, \bbM)$ be a fixed standard Borel atomless probability space. Define the law maps
	\begin{equation} \label{eq:lift intro}
		\begin{aligned}
			\iota_1\colon \XX_1 & \to \PP(\X_1), \qquad \quad
			\iota_2\colon \XX_2  \to \PP(\X_2), \qquad \quad
			\iota_{1,2}\colon \XX_1\times \XX_2 \to \PP(\X_1\times \X_2), \qquad
			\\
			\Z_1 & \mapsto (\Z_1)_\sharp \bbM 
			\hspace{1.73cm} \Z_2 \mapsto (\Z_2)_\sharp \bbM
			\hspace{1.95cm} (\Z_1,\Z_2) \mapsto (\Z_1,\Z_2)_\sharp \bbM.
		\end{aligned}
	\end{equation}
	where $\XX_i:=L^0\big( \rmQ, \cF_{\rmQ},\bbM; \X_i \big)$ is the set of measurable random variables (quotiented by the $\bbM$-a.e. equivalence relation) endowed with the topology of the convergence in probability. We can use these maps to lift our OT problem: define 
	\[\hat{\C}: \XX_1 \times \XX_2 \to [0,+\infty], \quad \hat{\C}(\Z_1,\Z_2):= \int_{\rmQ} \cost\big(\Z_1(q),\Z_2(q) \big) d\bbM(q),\]
	and use it to define the optimal transport problem 
	\[\hat{\calC}(\ttm_1,\ttm_2):= \min_{\frP\in\Gamma(\ttm_1,\ttm_2)} \left\{ \int_{\XX_1\times \XX_2} \hat{\C}(\Z_1,\Z_2) d\frP(\Z_1,\Z_2) \right\} \qquad \forall \ttm_i \in \PP(\XX_i).\]
	
	We can relate $\sfC$-concavity and $\hat{\C}$-concavity: using the notation of \eqref{eq: as intro}, we can define $\XX_{i,\tta_i} := \{\Z_i \in \XX_i : \int_{\rmQ} \tta_i(\Z_i(q)) d\bbM(q)<+\infty\}$, for $i=1,2$, which naturally restrict \eqref{eq:lift intro} to
	\begin{equation}\label{eq:lift intro tta}
		\begin{aligned}
			\iota_1\colon \XX_{1,\tta_1} & \to \PP_{\tta_1}(\X_1), \qquad \quad
			\iota_2\colon \XX_{2,\tta_2}  \to \PP_{\tta_2}(\X_2), \qquad \quad
			\iota_{1,2}\colon \XX_{1,\tta_1} \times \XX_{2,\tta_2} \to \PP_{\tta_1\oplus \tta_2}(\X_1\times \X_2).
		\end{aligned}
	\end{equation}
	At this point, we will show that a function $\phi_1:\PP_{\tta_1}(\X_1) \to [-\infty,+\infty)$ is $\C$-concave if and only if $\hat{\phi}_1:= \phi_1 \circ \iota_1 \colon \XX_{1,\tta_1} \to [-\infty,+\infty)$ is $\hat{\C}$-concave, and a similar statements holds for total $\C$-cyclical monotonicity (resp. total $\C$-superdifferential) and $\hat{\C}$-cyclical monotonicity (resp. $\hat{\C}$-superdifferential). These results are proven independently of the ones presented above and as a byproduct they give an alternative proof for the characterization of optimal random couplings.
	
	\begin{example}\label{ex: wass 3}
		In the context of Examples \ref{ex: wass} and \ref{ex: wass 2}, since $\X_1 = \X_2 = \X$ and $\tta_1 = \tta_2$, we only have one lifted space $\XX_{1,\tta_1} = \XX_{2,\tta_2} = \XX_p := L^p(\rmQ,\mathcal{F}_\rmQ,\bbM; \X)$, that is also a complete and separable metric space when endowed with the $L^p$-distance.
	\end{example}

	It is important to note that, as it is presented here, this strategy only works in the case that $\cost$ is continuous (still possibly infinitely valued). For the general case, we need to consider a different lifting procedure, see \eqref{eq:82} and \eqref{eq:83}, and then the same approach works.
	All the details about it are presented in Section \ref{sec: Lions}.

	\normalcolor
	
	\bigskip
	
	\noindent\textbf{Strict Monge formulation.} The usual Monge formulation that can be naturally associated to the OT cost $\calC$ is
	\begin{equation}\label{eq: monge cost intro}
		\calC_{M}(\ttM_1,\ttM_2):=\inf\left\{\int_{\PP(\X_1)} \C(\mu_1, \cT(\mu_1)) d\ttM_1(\mu_1)  \,: \, \cT\colon\PP(\X_1) \to \PP(\X_2) \text{ Borel s.t. }\cT_\sharp \ttM_1 = \ttM_2\right\}.
	\end{equation}
	However, this setting allows us to introduce a stricter version of it: taking a competitor for \eqref{eq: monge cost intro}, a natural question is whether the optimal cost $\sfC$ between $\mu_1$ and $\cT(\mu_1)$ can be realized, or at least approximated, by an optimal map $\rmT_{\mu_1} \colon \X_1 \to \X_2$. This argument leads to the \textit{strict Monge problem}:	
	\begin{equation}\label{eq: strict cost intro}
		\begin{aligned}
			\calC_{sM}(\ttM_1,\ttM_2):=\inf\Bigg\{\int_{\PP(\X_1)} \int_{\X_1} \cost(x_1,\rmT(x_1,\mu_1)) d\mu_1(x_1) d\ttM_1(\mu_1) \, : \quad
			\\
			\rmT:\X_1 \times \PP(\X_1) \to \X_2 \text{ Borel,\  }\cT_\sharp \ttM_1 = \ttM_2\Bigg\},
		\end{aligned}
	\end{equation}
	where in \eqref{eq: strict cost intro} the map $\cT\colon \PP(\X_1)\to\PP(\X_2)$ is defined as $\cT(\mu_1):= \rmT(\cdot,\mu_1)_\sharp \mu_1$. In general, not all the maps $\cT:\PP(\X_1) \to \PP(\X_2)$ can be represented in this way (see \cite{LS25}), and this gives that the problem \eqref{eq: strict cost intro} is usually defined on a more restrictive class than \eqref{eq: monge cost intro}. 
	
	Recall that, according to formula \eqref{eq: repr fully intro}, any competitor for the strict Monge formulation induces a fully deterministic random coupling.
	This shows the trivial inequalities $\calC \leq \calC_{M} \leq \calC_{sM}$.
	At this point, it is natural to ask: when $\calC(\ttM_1,\ttM_2) = \calC_{sM}(\ttM_1,\ttM_2)$? 
	Inspired by the general results by A. Pratelli \cite{pratelli2007equality}, we prove the following result (see Theorems \ref{thm: pratelli} and \ref{thm: pratelli2}).
	
	\begin{theorem}
		Let $\ttM_1\in \PP(\PP(\X_1))$ be atomless and concentrated on atomless measures. Then:
		\begin{enumerate}
			\item Assume $\cost:\X_1\times \X_2 \to [0,+\infty)$ is continuous and bounded, then for all $\ttM_2\in \PP(\PP(\X_2))$ it holds $\calC(\ttM_1,\ttM_2) = \calC_{sM}(\ttM_1,\ttM_2)$.
			\item Assume $\X_1 = \X_2 = \X$, $(\X,\rmd)$ is a complete and separable metric space. Let $p\in[1,+\infty)$. If furthermore $\ttM_1, \ttM_2 \in \PP_p(\PP_p(\X))$, then the same holds, that is for all $\varepsilon>0$ there exists $\rmT^\varepsilon:\X \times \PP(\X) \to \X$ such that, if $\cT^\varepsilon(\mu_1):= \rmT^\varepsilon(\cdot,\mu_1)_\sharp \mu_1$, 
			\[\cT^\varepsilon_\sharp \ttM_1 = \ttM_2 \quad \text{ and }\quad \int_{\PP(\X)}\int_{\X} \rmd^p(x_1,\rmT^\varepsilon(x_1,\mu_1)) d\mu_1(x_1) d\ttM_1(\mu_1) \leq \sfW_p^p(\ttM_1, \ttM_2) + \varepsilon.\]
		\end{enumerate}
	\end{theorem}
	\noindent Its proof strongly relies on Pratelli's theorem and on the characterization of optimal random couplings, but not on the Lions' lifting of Section \ref{sec: Lions}. 
	
	\bigskip
	
	\textbf{Brenier theorem for random measures on Banach spaces.} In \ref{subsec: BreBan}, not only we show that the values of the Kantorovich and the strict Monge problem coincide, but also that there exists a unique optimal random coupling and it is fully deterministic, under stronger assumptions on the base space and on the initial random measure. In particular, we put ourselves in the context of Examples \ref{ex: wass} and \ref{ex: wass 2}, with the additional assumption that $\X = \rmB$ is a separable Banach space endowed with a strictly convex norm $\|\cdot\|$ and $p\in(1,+\infty)$. 
	
	In this scenario, following Example \ref{ex: wass 3}, we have the great advantage that the lifted space is itself a separable Banach space with strictly convex norm, being the Bochner space $\cB:= L^p(\rmQ,\mathcal{F}_\rmQ, \bbM; \rmB)$. We use this structure to transfer `Gaussian' concepts from $\cB$ to $\PP_p(\rmB)$.
	
	\begin{enumerate}
		\item Following \cite[Chapter 6]{AGS08}, we say that a measure $\mu_1 \in \PP_p(\rmB)$ is Gaussian regular if $\mu_1(N) = 0$ for every Gaussian null set $N\subset \rmB$ (see \cite{Phelps78}).
		\item We say that a random measure $\ttM_1 \in \PP_p(\PP_p(\rmB))$ is super Gaussian regular if it is concentrated on transport regular measures (that is for $\ttM_1$-a.e. $\mu_1$ and every $\mu_2 \in \PP_p(\rmB)$ there exists a unique optimal coupling and it is deterministic) and $\ttM_1(\rmN) = 0$ for all $\rmN\subset \PP_p(\rmB)$ such that $\iota^{-1}(\rmN)$ is Gaussian null in $\cB$, where $\iota\colon \cB \to \PP_p(\rmB)$ is the law map defined in \eqref{eq:lift intro}.
	\end{enumerate}
	This approach is inspired by the one presented in \cite{PS25convex}, which was specific for the Hilbert setting with $p=2$. Then we can prove the following result.
	
	\begin{theorem}\label{thm: BreRM intro}
		Let $(\rmB, \|\cdot\|)$ be a Banach space with strictly convex norm. Let $p\in(1,+\infty)$ and $\ttM_1,\ttM_2 \in \PP_p(\PP_p(\rmB))$ such that $\ttM_1$ is super Gaussian regular. Then, there exists a unique optimal random coupling $\ttP \in \RGamma_{o,p}(\ttM_1,\ttM_2)$ and it has the form \eqref{eq: repr fully intro} for some Borel measurable $\rmT:\rmB\times \PP(\rmB) \to \rmB$. In particular
		\begin{equation*}
			\sfW_p^p(\ttM_1,\ttM_2) = \int_{\PP_p(\rmB)} \int_{\rmB} \|x_1 - \rmT(x_1,\mu_1)\|_{}^p d\mu_1(x_1)d\ttM_1(\mu_1), \quad \cT_\sharp \ttM_1 = \ttM_2, \quad \cT(\mu_1) = \rmT(\cdot,\mu_1)_\sharp\mu_1. 
		\end{equation*}
	\end{theorem}
	
	It is worth mentioning that the only previous version of this theorem is present in \cite{PS25convex, beiglbock2025brenier} and is specific of the case $p=2$ and $\rmB = \rmH$ Hilbert space. Thus, this result is new even in the finite dimensional case $\rmB = \R^d$ for every $d\geq 1$ (for any strictly convex norm), and the strict convexity of the norm is sharp, as we will show in Remark \ref{rem: infty norm}. Moreover, in the case $d=1$ a stronger result holds, as we can replace the hypotheses that $\ttM_1$ is concentrated on Gaussian regular measures with being concentrated on atomless measures, see Remark \ref{rem: R}. 
	
	Differently from what presented in the case of the general cost, this result strongly relies on Lions' lifting, because we are defining super Gaussian regularity through it (even if this definition is independent of the choice of the underlying probability space $(\rmQ, \cF_{\rmQ}, \bbM)$) and the proof relies on the optimal transport theory in the Bochner space $\mathcal{B}$. It would be interesting to understand how to develop a strategy for similar results that do not rely on the lifting. 
	
	Of course, the proof relies also on a version of the classic Brenier theorem on strictly convex Banach spaces. This result is folklore, but we did not find a clear reference for it and we prove it in \ref{subsub: banach monge}.

	\bigskip

	\noindent\textbf{Other literature.} The first paper studying the Monge problem on laws of random probability measures is \cite{Emami-Pass}, where the authors studied the uniqueness of optimal couplings in the case in which $\X_1 = \X_2$ is a Riemannian manifold and $\cost$ is the square distance. On the other hand, independently, \cite{beiglbock2025brenier, PS25convex} studied the strict Monge formulation when $\X_1 = \X_2$ is an Hilbert space with underlying cost function the square distance. In a setting similar to that of Section \ref{subsec: BreBan}, they introduced a class of measures $\ttM_1$ for which optimal random couplings are unique and induced by a solution to the strict Monge problem. Unlike the general Banach setting considered in Section 5, the Hilbert setting, with cost the squared distance, allows for weaker assumptions than Gaussian regularity: this is possible because the optimal Kantorovich potentials are convex and then differentiable out of a countable union of d.c.-hypersurfaces \cite{Zajicek79}. Convexity fails in general Banach spaces, and thus a different result on the differentiability of optimal Kantorovich potentials is needed in general Banach spaces (even for $p=2$ or in Hilbert but with $p\neq2$). Note that in \cite{beiglbock2025brenier}, the authors studied the iterated OT problem at deep $n\geq 2$, that is on the space $\PP_2(\dots \PP_2(\rmH) \dots)$, with applications to the Monge problem for the adapted Wasserstein distance between finite-time stochastic processes.
	
	Beyond optimal transport, many researchers have been recently attracted to study the space of laws of random probability measures: for applications in Bayesian statistics \cite{nguyen2016borrowing, catalano2024hierarchical, catalano2024wasserstein}; the study of evolution equations \cite{lacker2022superposition, pinzi2025first, Pinzi-Savare25, rehmeier2023linearization}, also for understanding geometric properties of the classic Wasserstein space from a metric measure perspective \cite{dello2022dirichlet, delloschiavo2024massive, pinziatomic2025}; fine structure of smooth/convex functionals and associated gradient flows \cite{bonet2025flowing-329, vauthier2025variational}.
	
	An overview on some of the aforementioned papers is currently in preparation as the Ph.D. thesis of the author \cite{pinzi2026+}.
	
	There is a vast research area on optimal transport between random measures: for example, in \cite{huesmann2016optimal}, the author studies the optimal transport between two measure-valued random variables $\mu_1^{\omega}$ and $\mu_2^{\omega}$, assuming equivariant properties, due to the fact that they have infinite mass almost surely (see also
	\cite{last2009invariant, huesmann2013optimal, last2015construction, erbar2025optimal}). Moreover, these kinds of settings are also important for studying the so-called \textit{random matching problems}, see e.g. \cite{ambrosiostratrevisan, goldman2021convergence, huesmann2024asym}.

	\bigskip
	
	\noindent\textbf{Plan of the paper.} In Section \ref{sec: prel} we recall some preliminaries about the topology and Wasserstein metric on the space of laws of random probability measures, and the main classic results of optimal transport theory. In Section \ref{sec: OT gen cost}, we first fix the setting, and then prove the above mentioned characterization of optimal random couplings. In particular, we give the definition of total $\sfC$-cyclical monotonicity and total $\sfC$-cyclical superdifferential, studying their structure in Propositions \ref{prop: char tot superdiff} and \ref{prop: char tot mon}. In Section \ref{sec: Lions}, we first recall some results about convergence in probability, and then define the Lions' lifting for representing the space of probability measures over a Polish space. We then show how to lift the OT problem and how it relates to \eqref{eq: OTRM rand intro}, which in particular will provide an alternative proof of the characterization of the previous section. Finally, in Section \ref{sec: strict}, we introduce the strict Monge formulation and prove the results presented above.

	\bigskip

	\textbf{Acknowledgements.} The author warmly thanks Gudmund Pammer for stimulating discussions on this topic. The author is grateful to Giuseppe Savar\'e for carefully reading a preliminary version of this paper and for giving valuable suggestions on its improvement. He also thanks Hugo Lavenant for providing a preliminary version of \cite{LS25}.
	The author is member of GNAMPA(INdAM). The author has been partially supported by the European Research Council (ERC) under the European Union’s Horizon Europe research and innovation programme (grant agreement No. 101200514, project acronym OPTiMiSE).
	Views and opinions expressed are however those of the author(s) only and do not necessarily reflect those of the European Union or the European Research Council Executive Agency. Neither the European Union nor the granting authority can be held responsible for them.

	\section{Preliminaries}\label{sec: prel}

	\subsection{Narrow topology}
	Let $(\Y,\tau)$ be a Polish space. We will denote with $\mathcal{B}(\Y)$ its Borel $\sigma$-algebra and with $\PP(\Y)$ the probability measures defined over $\mathcal{B}(\Y)$. We endow $\PP(\Y)$ with the narrow topology, i.e. the coarsest topology that makes continuous the functions \(\mu\mapsto \int_{\Y} \phi d\mu\) for all $\phi \in C_b(\Y)$ (continuous and bounded functions). We will simply write $\mu_n \to \mu$ if the sequence of probability measures $\mu_n \in \PP(\Y)$ converges to $\mu\in \PP(\Y)$ in the narrow topology. Given a map $f:\Y_1 \to \Y_2$ between two measurable spaces, for all $\mu\in \PP(\Y_1)$ we will denote by $f_\sharp \mu\in \PP(\Y_2)$ the push-forward measure. Moreover, given a product space $\Y_1 \times \dots \times \Y_k$, we denote by $\operatorname{pr}^i$ the projection on the $i$-th coordinate, for $i = 1,\dots,k$.
	
	The previous construction can be iterated: in fact, the set of probability measures $\PP(\Y)$, endowed with the narrow topology, is a Polish space, so that we use the previous setting to define the space $\PP(\PP(\Y))$ with its narrow topology. 
	We will denote
	\begin{equation}\label{eq: overline RM}
		\text{for all }\ttM \in \PP(\PP(\Y)), \quad \overline{\ttM}\in \PP(\Y \times \PP(\Y)) \ \text{ defined as } 
		\overline{\ttM}:= \int_{\PP(\Y)} \mu \otimes \delta_\mu d\ttM(\mu).
	\end{equation}
	Given a map $f\colon\Y_1 \to \Y_2$ we define by induction the iterated push-forward as $f_{\sharp\sharp}\colon \PP(\PP(\Y_1))\to \PP(\PP(\Y_2))$ as $f_{\sharp\sharp} := (f_\sharp)_\sharp$ (see e.g. \cite[Proposition D.8]{Pinzi-Savare25}).

	\subsection{$\texorpdfstring{\operatorname{OT}_{\mathsf{C}}}{}$- problem, $\texorpdfstring{\mathsf{C}}{}$-concave functions and $\texorpdfstring{\mathsf{C}}{}$-superdifferentiability}\label{subsec: C-conc theory}
	Here we recall some classic results on optimal transport. For more detailed expositions we refer to \cite{AGS08, Santambrogio15, Villani09}. 
	Let $\Y_1$ and $\Y_2$ be Polish spaces, and $\C\colon\Y_1\times \Y_2 \to [0,+\infty]$ a (non-identically $+\infty$) cost function. Given $\mu_i \in \PP(\Y_i)$, the optimal transport problem is 
	\begin{equation}\tag{$\operatorname{OT}_{\mathsf{C}}$}\label{eq:00}
		\min\left\{ \int_{\Y_1\times\Y_2} \C(y_1,y_2) d\pi(\mu_1,\mu_2) \ : \ \pi \in \Gamma(\mu_1,\mu_2) \right\},
	\end{equation}
	where $\Gamma(\mu_1,\mu_2)$ is the set of couplings $\pi \in \PP{}(\Y_1\times \Y_2)$ such that $\operatorname{pr}^i_\sharp = \mu_i$, for $i=1,2$. The following theorem guarantees mild sufficient conditions for which the minimum exists and it is finite, and it can be proven by the Direct Method of the Calculus of Variations and the Prohorov's theorem.
	\begin{theorem}\label{thm:OT_C}
		Assume that $\C$ is lower semicontinuous. 
		Then, there exists $\pi \in \Gamma(\mu_1,\mu_2)$ realizing the (possibly infinite) minimum in \eqref{eq:00}. The set of optimal couplings is denoted as $\Gamma_{o,\C}(\mu_1,\mu_2)$.
	\end{theorem}
	
	We can actually say a lot more on the structure of optimal couplings, exploiting the dual problem of \eqref{eq:00} and the theory of $\C$-concave functions.
	
	\begin{definition}\label{def:C-conc} Let $\Y_1',\Y_2'$ be two sets and $\C\colon\Y_1'\times \Y_2' \to [0,+\infty)$.
		\begin{enumerate}
			\item[$(1)$] We say that $\Gamma \subset \Y_1'\times \Y_2'$ is $\C$-cyclically monotone if for all $k \geq 1$, for all $\sigma$ permutation of $\{1,\dots,k\}$ and $(y_{1,i}, y_{2,i}) \in \Gamma$, $i = 1,\dots,k$, it holds
			\begin{equation}\label{eq:10}
				\sum_{i=1}^k \C(y_{1,i},y_{2,\sigma(i)}) \geq  \sum_{i=1}^k \C(y_{1,i},y_{2,i}).
			\end{equation}
			Note that, by decomposition of permutation in cycles, it suffices to check that for all $k\geq 1$ and $(y_{1,i}, y_{2,i})\in \Gamma$ it holds
			\[\sum_{i=1}^k \C(y_{1,i},y_{2,i-1}) \geq  \sum_{i=1}^k \C(y_{1,i},y_{2,i}),\]
			where we use the convention $y_{2,0} = y_{2,k}$.
			\item[$(2)$] A function $\phi_1\colon\Y_1' \to [-\infty,+\infty)$ is said to be $\C$-concave if it is not identically $-\infty$ and there exists $A\subset \Y_2' \times \R$ such that 
			\begin{equation}\label{eq:11}
				\phi_1(y_1) = \inf_{(y_2,\alpha) \in A} \C(y_1,y_2) -\alpha.
			\end{equation}
			Analogously, we can define $\C$-concavity for a function $\phi_2\colon\Y_2' \to [-\infty,+\infty)$.
			\item[$(3)$] Given a function $\phi_1\colon\Y_1' \to [-\infty,+\infty)$, we define its $\C$-conjugate function $\phi_1^{\C}\colon\Y_2' \to [-\infty,+\infty]$
			\begin{equation}\label{eq:12}
				\phi_1^{\C}(y_2) := \inf_{y_1\in\Y_1'} \C(y_1,y_2) - \phi(y_1).
			\end{equation}
			Similarly, for $\phi_2\colon\Y_2' \to [-\infty,+\infty)$, we define $\phi_2^{\C}\colon \Y_1' \to [-\infty,+\infty)$.
			\item[$(4)$] If $\phi_1\colon\Y_1' \to [-\infty,+\infty)$ and $y_1 \in \operatorname{dom}\phi:= \{\phi_1>-\infty\}$, its $\C$-superdifferential is defined as 
			\begin{equation}\label{eq:13}
				\begin{gathered}
					\partial^+_{\C}\phi_1 := \left\{(y_1,y_2)\in \operatorname{dom}\phi\times \Y_2'  :  \C(y_1,y_2) - \phi_1(y_1) \leq \C(y_1',y_2) - \phi_1(y_1') \quad \forall y_1' \in \Y_1'\right\},
					\\
					\partial^+_\C\phi_1(y_1):= \left\{y_2  : (y_1,y_2)\in \partial^\C\phi_1\right\}.
				\end{gathered}
			\end{equation}
		\end{enumerate}
	\end{definition}
	
	With these definitions, it is not hard to show that a function $\phi_1$ is $\C$-concave if and only if $\phi_1^{\C\C} = \phi_1$. Moreover the $\C$-superdifferential $\partial^+_\C\phi_1$ is $\C$-cyclically monotone. The celebrated Rockafellar--Ruschendorf theorem gives the opposite direction.
	
	\begin{theorem}\label{thm:RR}
		Assume that $\Gamma \subset \Y_1'\times \Y_2'$ is $\C$-cyclically monotone. Then, there exists a $\sfC$-concave function $\phi\colon\Y_1' \to [-\infty,+\infty)$ such that $\Gamma \subset \partial^+_{\C}\phi$.
	\end{theorem}
	
	The next result is again classic in OT theory and characterizes the optimal couplings exploiting the theory of $\C$-concave functions, see for example \cite[Theorem 4.2]{ambrosio2021lectures}.
	
	\begin{theorem}\label{thm:char opt plans}
		Let $\Y_1$ and $\Y_2$ be Polish spaces, and $\mu_i \in \PP(\Y_i)$, $i=1,2$. Let $\C\colon\Y_1\times \Y_2 \to[0,+\infty]$ be lower semicontinuous and assume there exist $\rmA_i \colon \Y_i \to [0,+\infty]$ such that 
		\begin{equation}\label{eq:01}
			\rmA_i \in L^1(\mu_i) \quad \text{for } i=1,2, \qquad \C(y_1,y_2) \leq \rmA_1(y_1) + \rmA_2(y_2) \quad \forall (y_1,y_2)\in \Y_1\times \Y_2.
		\end{equation}
		Let $\Y_i':= \rmA_i^{-1}([0,+\infty))$, so that $\C\colon\Y_1'\times \Y_2' \to [0,+\infty)$, and let $\pi\in\Gamma(\mu_1,\mu_2)$. The following are equivalent:
		\begin{enumerate}
			\item $\pi\in \Gamma_{o,\C}(\mu_1,\mu_2)$;
			\item $\pi$ is concentrated on a $\C$-cyclically monotone set $\Gamma \subset \Y_1'\times \Y_2'$;
			\item there exists $\phi\colon\Y_1'\to[-\infty,+\infty)$ such that $\phi \in L^1(\mu_1)$, $\phi^{\C}\in L^1(\mu_2)$ and
			\[\int_{\Y_1'} \phi(y_1) d\mu_1(y_1) + \int_{\Y_2'} \phi^{\C}(y_2) d\mu_2(y_2) = \int_{\Y_1\times \Y_2} \C(y_1,y_2) d\pi(y_1,y_2),\]
			i.e. $\pi$ is concentrated over $\partial^+_{\C}\phi$.
		\end{enumerate}
	\end{theorem}
	
	\subsection{The \texorpdfstring{$L^p$}{}-Wasserstein space}\label{subsec: wass}
	
	Let $(\Y,\rmd)$ be a separable and complete metric space, and $p\in[1,+\infty)$. Using the notation of the previous subsection, let $\Y_1 = \Y_2 = \Y$ and $\C(y_1,y_2) = \rmd(y_1,y_2)^p$. Then, the $L^p$-Wasserstein distance between $\mu_1,\mu_2 \in \PP(\Y)$ corresponds to the $p$-root of the value \eqref{eq:00}. We will denote it as $\sfw_{p,\rmd}(\mu_1,\mu_2) \in [0,+\infty]$ and the set of optimal couplings as $\Gamma_{o,p}(\mu_1,\mu_2)$. To ensure that it is finite (and in particular, a distance) we restrict to the set of probability measures with finite $p$-moment
	\begin{equation}\label{eq: p-moment}
		\PP_p(\Y) := \left\{ \mu \in \PP(\Y) \ : \ \int_{\Y}\rmd^p(y,y_0) d\mu(y)<+\infty \text{ for some (and then all) }y_0 \in \Y \right\}.
	\end{equation}
	Notice that, with the notation of the previous sections, $\PP_p(\Y)$ coincides with $\PP_{\tta_p}(\Y)$ with $\tta_p(y):= \rmd^p(y,y_0)$. Recalling that the space $(\PP_p(\Y),\sfw_{p,\rmd})$ is a complete and separable metric space, we are allowed to iterate the previous construction: replacing the space $(\Y,\rmd)$ with $(\PP_p(\Y),\sfw_{p})$, we use it to define the $L^p$-Wasserstein-on-Wasserstein space 
	\begin{equation}\label{eq:50}
		\big(\PP_p(\PP_p(\Y)), \sfW_{p,\rmd}\big) \quad \text{ with }\quad \sfW_{p,\rmd} := \sfw_{p,\sfw_{p,\rmd}}.
	\end{equation}
	When the distance $\rmd$ is clear from the context, we simply write $\sfw_p$ and $\sfW_{p}$. See also Examples \ref{ex: wass} and \ref{ex: wass 2}.

	\section{Optimal transport between laws of random probability measures}\label{sec: OT gen cost}
	Let $\X_1$ and $\X_2$ be two Polish spaces, and $\cost\colon\X_1\times \X_2 \to [0,+\infty]$ a (non-identically $+\infty$) lower semicontinuous cost function. For all $\mu_i \in \PP(\X_i)$, denote
	\begin{equation}\label{eq:60}
		\C(\mu_1,\mu_2) := \min \left\{ \int_{\X_1\times \X_2} \cost(x_1,x_2) d\pi(x_1,x_2) \ : \ \pi \in \Gamma(\mu_1,\mu_2) \right\}.
	\end{equation}

	\begin{lemma}\label{lemma: lsc C}
		The map $\C\colon\PP(\X_1)\times \PP(\X_2) \to [0,+\infty]$ is non-indentically $+\infty$ and lower semicontinuous.
	\end{lemma}
	
	\begin{proof}
		The first part follows noticing that, if $\cost(x_1,x_2)<+\infty$, then $\C(\delta_{x_1},\delta_{x_2}) = \cost(x_1,x_2)$. On the other hand, let $\mu_{1,n} \to \mu_1$ and $\mu_{2,n} \to \mu_2$ in the respective weak topologies and consider $\pi_n \in \Gamma_{o,\cost}(\mu_{1,n},\mu_{2,n})$. Without loss of generality, assume that $\liminf \C(\mu_{1,n},\mu_{2,n}) = \lim \C(\mu_{1,n},\mu_{2,n})$. The sequences $\{\mu_{1,n}\}_{n\in \N} $ and $\{\mu_{2,n}\}_{n\in \N}$ are tight, which gives that $\{\pi_{n}\}_{n\in \N}$ is tight as well, and in particular there exists $\pi\in \Gamma(\mu_1,\mu_2)$ such that $\pi_{n_k} \to \pi$ narrowly, for some subsequence. Then 
		\[\C(\mu_1,\mu_2) \leq \int_{\X_1\times \X_2} \cost d\pi\leq \lim_{k\to +\infty} \int_{\X_1\times \X_2} \cost d\pi_{n_k} = \liminf_{n\to+\infty} \C(\mu_{1,n},\mu_{2,n}). \qedhere\]
	\end{proof}
	
	Thus, we can iterate the optimal transport construction and consider the problem
	\begin{equation}\label{eq:61}
		\calC(\M_1,\M_2) := \min \left\{ \int_{\PP(\X_1)\times \PP(\X_2)} \C(\mu_1,\mu_2) d\Pi(\mu_1,\mu_2) \ : \ \Pi \in \Gamma(\M_1,\M_2) \right\}.
	\end{equation}
	The particular structure of this problem allows us to rephrase it as a minimization problem over the so-called \textit{(laws of) random couplings}. We say that $\ttP \in \RGamma (\M_1,\M_2) \subset \PP(\PP(\X_1\times \X_2))$ if 
	\[\operatorname{pr}^1_{\sharp\sharp}\ttP = \M_1, \quad \operatorname{pr}^2_{\sharp\sharp}\ttP = \M_2.\]
	Any random coupling $\ttP \in \RGamma (\M_1,\M_2)$ induces a usual coupling 
	\begin{equation}\label{eq:70}
		\Pi := \big(\operatorname{pr}^1_\sharp, \operatorname{pr}^2_\sharp\big)_\sharp \ttP \in \PP(\PP(\X_1)\times \PP(\X_2)).
	\end{equation}
	
	\begin{lemma}\label{lemma: rand coupl}
		Let $\ttP \in \RGamma (\M_1,\M_2)$ and $\Pi$ defined as in \eqref{eq:70}. Then $\Pi \in \Gamma(\M_1,\M_2)$ and 
		\begin{equation}\label{eq:71}
			\int_{\PP(\X_1)\times \PP(\X_2)} \C(\mu_1,\mu_2) \Pi(\mu_1,\mu_2) \leq \int_{\PP(\X_1\times \X_2)} \int_{\X_1\times \X_2} \cost(x_1,x_2) d\pi(x_1,x_2) d\ttP(\pi).
		\end{equation}
		Moreover, the following equality holds and the minimum (possibly infinite) is always attained:
		\begin{equation}\label{eq:72}
			\calC(\M_1,\M_2) = \min\left\{  \int_{\PP(\X_1\times \X_2)} \int_{\X_1\times \X_2} \cost(x_1,x_2) d\pi(x_1,x_2) d\ttP(\pi) \ : \ \ttP \in \RGamma  (\M_1,\M_2) \right\}.
		\end{equation}
		If $\mathcal{C}(\ttM_1,\ttM_2)<+\infty$, then a random coupling $\ttP$ is optimal, and we write $\ttP \in \RGamma _{o,\calC}$, if and only if $(\operatorname{pr}^1_\sharp,\operatorname{pr}^2_\sharp)_\sharp \ttP \in \Gamma_{o,\C}(\M_1,\M_2)$ and it is concentrated on the set of optimal couplings $\PP_{o,\C}(\X_1\times \X_2)$.
	\end{lemma}
	
	\begin{proof}
		By the very definition of $\Pi$, it is a coupling between $\ttM_1$ and $\ttM_2$, and \eqref{eq:71} follows:
		\begin{align*}
			\int \C(\mu_1,\mu_2) \Pi(\mu_1,\mu_2) = \int \C(\operatorname{pr}^1_\sharp \pi, \operatorname{pr}^2_\sharp \pi) d\ttP(\pi) \leq \int \int \cost(x_1,x_2) d\pi(x_1,x_2) d\ttP(\pi).
		\end{align*}
		Thus, the $\leq$ in \eqref{eq:72} follows. Regarding the converse inequality, consider the map 
		\[(\operatorname{pr}^1_\sharp ,\operatorname{pr}^2_\sharp) \colon \PP_{o,\C}(\X_1\times \X_2) \to \PP(\X_1)\times \PP(\X_2).\]
		Thanks to Theorem \ref{thm:OT_C}, it is a surjective map between Borel measurable sets, and in particular there exists a right inverse map $G\colon \PP(\X_1)\times \PP(\X_2) \to \PP_{o,\C}(\X_1\times \X_2)$ that is universally measurable (see \cite[Theorem 6.9.1]{Bogachev07}). Then, given $\Pi \in \Gamma_{o,\C}(\M_1,\M_2)$, we may consider $\ttP := G_\sharp \Pi$, which gives
		\[
		\begin{aligned}
			\calC(\M_1,\M_2) = & \int \C(\mu_1,\mu_2) d\Pi(\mu_1,\mu_2) = \int \int \cost(x_1,x_2) d\big(G(\mu_1,\mu_2)\big)(x_1,x_2)  d\Pi(\mu_1,\mu_2) 
			\\
			= & \int_{\PP(\X_1\times \X_2)} \int_{\X_1\times \X_2} \cost(x_1,x_2) d\pi(x_1,x_2) d\ttP(\pi).
		\end{aligned}\] 
		The last part follows by noticing that in \eqref{eq:71} the equality holds if and only if $\ttP$ is concentrated on $\PP_{o,\C}(\X_1\times \X_2)$.
	\end{proof}
	
	\subsection{Total $\texorpdfstring{\C}{}$-cyclical monotonicity, total $\texorpdfstring{\C}{}$-superdifferential and characterization of optimal random couplings}
	In this section, we give a characterization of optimal random couplings in terms of the optimal Kantorovich potentials. We work in the following setting: let $\X_1$ and $\X_2$ be two Polish spaces, $\ttM_i \in \PP(\PP(\X_i))$ for $i = 1,2$, $\cost\colon\X_1\times \X_2 \to [0,+\infty]$ a {lower semicontinuous} function satisfying
	\begin{equation}\label{eq:100}
		\cost(x_1,x_2) \leq \tta_1(x_1) + \tta_2(x_2) \quad \forall (x_1,x_2) \in \X_1 \times \X_2
	\end{equation}
	for some Borel measurable maps $\tta_i\colon\X_i \to [0,+\infty]$, $i=1,2$.

	\begin{definition}
		We denote
		\begin{equation}\label{eq:110}
			\begin{gathered}
				\ttA_i(\mu_i) := \int_{\X_i} \tta_i d\mu_i, \qquad\PP_{\tta_i}(\X_i) := \left\{ \mu_i \in \PP(\X_i) \ : \ \ttA_i(\mu_i)<+\infty\right\},
				\\
				\PP_{\tta_1\oplus\tta_2}(\X_1\times \X_2) := \left\{ \pi \in \PP(\X_1\times \X_2) \ : \ \operatorname{pr}^i_\sharp \pi \in \PP_{\tta_i}(\X_i), \ i = 1,2 \right\}.
			\end{gathered}
		\end{equation}
	\end{definition}
	
	The function $\ttA_i\colon\PP(\X_i) \to [0,+\infty]$ is Borel measurable, see e.g. \cite[Lemma D.1]{Pinzi-Savare25}. Consequently, the sets $\PP_{\tta_i}(\X_i)\subset \PP(\X_i)$ and $\PP_{\tta_1\oplus\tta_2}(\X_1\times \X_2)$ are Borel in their respective narrow topology.

	\begin{assumption}\label{ass: integr}
		The random measures $\ttM_i \in \PP(\PP(\X_i))$ satisfy 
		\[\int_{\PP(\X_i)} \ttA_i(\mu_i) d\ttM_i(\mu_i)<+\infty.\]
	\end{assumption}

	The general results presented in Section \ref{subsec: C-conc theory} will be exploited in the following setting: for $i=1,2$
	\[\Y_i = \PP(\X_i), \quad \rmA_i = \ttA_i, \quad \Y_i' = \PP_{\tta_i}(\X_i), \quad \text{function $\C$ defined in \eqref{eq:60}}.\]

	\noindent Note that under Assumption \ref{ass: integr}, we have:
	\begin{enumerate}
		\item $\ttM_i$ is concentrated on $\PP_{\tta_i}(\X_i)$ for $i=1,2$;
		\item for all $(\mu_1,\mu_2) \in \PP_{\tta_1}(\X_1)\times \PP_{\tta_2}(\X_2)$ it holds $\C(\mu_1,\mu_2)\leq \ttA_1(\mu_1) + \ttA_2(\mu_2) <+\infty$.
	\end{enumerate}

	We now relate this setting with the Kantorovich potentials, extending the study of $\C$-concave theory in this scenario. In particular, we introduce the notion of \textit{total $\C$-superdifferential} and \textit{total $\C$-cyclical monotonicity}. The terminology is justified by \cite{CSS3,PS25convex}, which treat with the setting where $\X_1 = \X_2$ is an Hilbert space and $\cost(x_1,x_2) = |x_1 - x_2|^2$.

	\begin{definition}[Total $\C$-superdifferential]\label{def: total super}
		Let $\phi_1\colon\PP_{\tta_1}(\X_1) \to[-\infty,+\infty)$. We define its total $\C$-superdifferential as
		\begin{equation}
			\begin{aligned}
				\partial_{\ttt, \C}^+\phi_1 := \bigg\{ \pi \in \PP_{\tta_1\oplus \tta_2}(\X_1\times \X_2) \ : \ \forall \theta \in \PP(\X_1\times \X_2 \times \X_1) \text{ s.t. }(\operatorname{pr}^1,\operatorname{pr}^2)_\sharp \theta = \pi, \ \operatorname{pr}^3_\sharp\theta \in \PP_{\tta_1}(\X_1),&
				\\
				\int_{\X_1\times \X_2 \times \X_1} \cost(x_1,x_2) - \cost(x_1',x_2) d\theta(x_1,x_2,x_1') \leq \phi_1 \big(\operatorname{pr}^1_\sharp \theta\big) - \phi_1 \big( \operatorname{pr}^3_\sharp\theta \big)
				& \bigg\}.
			\end{aligned}
		\end{equation}
	\end{definition}
	
	\begin{definition}[Total $\C$-cyclical monotonicity]\label{def: total cycl mon}
		We say that $\rmF \subset \PP_{\tta_1\oplus \tta_2}(\X_1\times \X_2)$ is a total $\C$-cyclically monotone subset if for all $N\geq 1$, for all $\theta \in \PP\big((\X_1\times \X_2)^N\big)$ such that $\operatorname{pr}^n_\sharp\theta \in \rmF$ with $n\leq N$, and for all $\sigma$ permutation of $\{1,\dots,N\}$ it holds
		\begin{equation}
			\int_{(\X_1\times \X_2)^N} \sum_{i=1}^N \cost(x_{1,i},x_{2,i}) d\theta\leq   \int_{(\X_1\times \X_2)^N} \sum_{i=1}^N \cost(x_{1,i},x_{2,\sigma(i)}) d\theta
		\end{equation}
	\end{definition}
	
	The next two propositions aim at relating the notion of total $\C$-superdifferential (resp. total $\C$-cyclical monotonicity) with the classic $\C$-superdifferential (resp. $\C$-cyclical monotonicity).
	\begin{proposition}\label{prop: char tot superdiff}
		Let $\phi_1\colon\PP_{\tta_1}(\X_1) \to [-\infty,+\infty)$ be $\C$-concave and not identically $-\infty$. Then:
		\begin{enumerate}
			\item It holds
			\begin{equation}
				\partial_{\ttt,\C}^+ \phi_1 = \left\{ \pi \in \PP_{\tta_1\oplus \tta_2}(\X_1\times \X_2) \ : \ \pi \in \Gamma_{o,\cost}(\operatorname{pr}^1_\sharp\pi,\operatorname{pr}^2_\sharp \pi) , \ (\operatorname{pr}^1_\sharp \pi, \operatorname{pr}^2_\sharp \pi ) \in \partial_{\C}^+\phi_1\right\}.
			\end{equation}
			In particular, all $\pi \in \partial_{\ttt,\C}^+\phi_1$ are $\C$-optimal and 
			$
			\big( \operatorname{pr}^1_\sharp, \operatorname{pr}^2_\sharp \big)  \partial_{\ttt,\C}^+\phi_1 = \partial_{\C}^+\phi_1.
			$
			\item $\pi \in \partial_{\ttt,\C}^+\phi_1$ if and only if 
			\begin{equation}\label{eq: eq}
				\phi_1\big( \operatorname{pr}^1_\sharp \pi\big) + \phi_1^{\C}\big(\operatorname{pr}^2_\sharp \pi \big) = \int_{\X_1 \times \X_2} \cost(x_1,x_2) d\pi(x_1,x_2).\end{equation}
			
			\item $\partial_{\ttt,\C}^+\phi_1$ is totally $\C$-cyclically monotone.
		\end{enumerate}
	\end{proposition}
	
	\begin{proof}
		\textit{1.} We first prove the inclusion $\supseteq$. Let $\pi$ be as in the right hand side, and consider a generic $\theta$ as in the definition of $\partial_{\ttt,\C}^+\phi_1$. Define $\mu_1':= \operatorname{pr}^3_\sharp \theta$, then
		\begin{align*}
			\int_{\X_1\times \X_2 \times \X_1}&  \cost(x_1,x_2) - \cost(x_1',x_2) d\theta(x_1,x_2,x_1')  =  \C(\mu_1,\mu_2) - \int_{\X_1\times \X_2 \times \X_1}  \cost(x_1',x_2) d\theta(x_1,x_2,x_1')
			\\
			\leq & \C(\mu_1,\mu_2) - \C(\mu_1',\mu_2) \leq \phi_1(\mu_1) - \phi_1(\mu_1') = \phi \big(\operatorname{pr}^1_\sharp \theta\big) - \phi \big( \operatorname{pr}^3_\sharp\theta \big).
		\end{align*}
		Regarding the other inclusion, we first prove that any $\pi \in \partial_{\ttt,\C}^+\phi_1$ is $\C$-optimal. By contradiction, assume that it is not optimal and take $\pi' \in \Gamma_{o,\cost}(\mu_1,\mu_2)$, where $\mu_i := \operatorname{pr}^i_\sharp \pi$, for $i=1,2$. Then, by the gluing lemma \cite[Lemma 5.3.2]{AGS08} we can find a measure $\theta \in \PP(\X_1\times \X_2 \times \X_1)$ such that $(\operatorname{pr}^1,\operatorname{pr}^2)_\sharp \theta = \pi$ and $(\operatorname{pr}^3,\operatorname{pr}^2)_\sharp\theta = \pi'$. Then, since $\operatorname{pr}^1_\sharp \theta = \operatorname{pr}^3_\sharp \theta$, we reach a contradiction:
		\begin{align*}
			0 \geq & \int_{\X_1\times \X_2 \times \X_1}  \cost(x_1,x_2) - \cost(x_1',x_2) d\theta(x_1,x_2,x_1') 
			\\
			= & \int_{\X_1\times \X_2}  \cost(x_1,x_2) d\pi(x_1,x_2) - \int_{\X_1\times \X_2}\cost(x_1',x_2) d\pi'(x_1',x_2) > 0.
		\end{align*}
		Now, we are left to show that $\mu_2 \in \partial_{\C}^+\phi_1(\operatorname{pr}^1_\sharp \pi)$. For all $\mu_1' \in \PP_{\tta_1}(\X_1)$, consider $\pi' \in \Gamma_{o,\cost}(\mu_1',\mu_2)$, and, again by the gluing lemma, build $\theta \in \PP(\X_1\times \X_2\times \X_1)$ such that $(\operatorname{pr}^1,\operatorname{pr}^2)_\sharp \theta = \pi$ and $(\operatorname{pr}^3,\operatorname{pr}^2)_\sharp\theta = \pi'$. Then
		\begin{align*}
			\C(\mu_1,\mu_2) - \C(\mu_1',\mu_2) = \int_{\X_1\times \X_2 \times \X_1}  \cost(x_1,x_2) - \cost(x_1',x_2) d\theta(x_1,x_2,x_1') \leq \phi_1(\mu_1) - \phi_1(\mu_1').
		\end{align*}
		
		\textit{2.} From point 1., $\pi \in \partial_{\ttt,\C}^+\phi_1$ if and only if it is optimal and the couple $(\mu_1,\mu_2) := (\operatorname{pr}^1_\sharp,\operatorname{pr}^2_\sharp)\pi$ belongs to $\partial_\C^+\phi_1$. This is equivalent to 
		\[\C(\mu_1,\mu_2) = \int_{\X_1\times \X_2} \cost(x_1,x_2) d\pi(x_1,x_2) \quad \text{ and } \quad \phi_1(\mu_1) + \phi_1^\C(\mu_2) = \C(\mu_1,\mu_2).\]
		Since the inequalities 
		\(\phi_1(\mu_1)+ \phi_1^\C(\mu_2) \leq \C(\mu_1,\mu_2) \leq \int \cost d\pi\) are always true, the latter is equivalent to \eqref{eq: eq}
		
		\textit{3.} Fix $N\geq 1$, a permutation $\sigma$ and $\theta \in \PP\big((\X_1\times \X_2)^N\big)$ such that $\operatorname{pr}^n_\sharp \theta \in \partial_{\ttt,\C}^+\phi_1$. Thanks to the first point and the fact that $\partial_{\C}^+\phi_1$ is $\C$-cyclically monotone, we have 
		\begin{align*}
			\int_{(\X_1\times \X_2)^N}& \sum_{i=1}^N \cost(x_{1,i},x_{2,i}) d\theta = \sum_{i=1}^N \C(\operatorname{pr}^{1,i}_\sharp\theta,\operatorname{pr}^{2,i}_\sharp \theta) 
			\\
			\leq  & \sum_{i=1}^N \C(\operatorname{pr}^{1,i}_\sharp\theta,\operatorname{pr}^{2,\sigma(i)}_\sharp \theta) \leq  \int_{(\X_1\times \X_2)^N} \sum_{i=1}^N \cost(x_{1,i},x_{2,\sigma(i)}) d\theta
		\end{align*}
	\end{proof}

	\begin{proposition}\label{prop: char tot mon}
		Let $\rmF \subset \PP_{\tta_1\oplus \tta_2}(\X_1\times \X_2)$ be a total $\C$-cyclically monotone subset. Then:
		\begin{enumerate}
			\item All $\pi \in \rmF$ are $\C$-optimal. 
			\item The set $(\operatorname{pr}^1_\sharp,\operatorname{pr}^2_\sharp)\rmF \subset \PP_{\tta_1}(\X_1) \times \PP_{\tta_2}(\X_2)$ is $\C$-cyclically monotone; 
			\item There exists a $\C$-concave function $\phi_1\colon\PP_{\tta_1}(\X_1)\to [-\infty,+\infty)$ such that $\rmF\subseteq \partial_{\ttt,\C}^+\phi_1$.
		\end{enumerate}
		In particular, for all $\mathfrak{F} \subset \PP_{\tta_1}(\X_1)\times \PP_{\tta_2}(\X_2)$ $\C$-cyclically monotone, the set 
		\begin{equation}\label{eq:111}
			\overline{\rmF} := \left\{ \pi \in \PP_{\tta_1\oplus \tta_2}(\X_1\times \X_2) \ : \ \pi \in \PP_{o,\C}(\X_1\times \X_2), \ \big( \operatorname{pr}^1_\sharp \pi, \operatorname{pr}^2_\sharp \pi \big) \in \mathfrak{F} \right\}
		\end{equation}
		is a total $\C$-cyclically monotone subset.
	\end{proposition}
	
	\begin{proof}
		\textit{1.} The argument is the same of the previous proof. Let $\pi \in \rmF$ and take any $\pi' \in \Gamma_{o,\cost}(\mu_1,\mu_2)$, where $\mu_i := \operatorname{pr}^i_\sharp \pi$, for $i=1,2$. Iterating the gluing lemma \cite[Lemma 5.3.2]{AGS08}, for all $N\in \N$ we can find a measure ${\theta} \in \PP\big((\X_1\times \X_2)^N\big)$ such that \[(\operatorname{pr}^{1,i},\operatorname{pr}^{2,i})_\sharp {\theta} = \pi\quad \text{and}\quad(\operatorname{pr}^{1,j+1},\operatorname{pr}^{2,j})_\sharp{\theta} = \pi'\]
		for all $i\in \{1,\dots,N\}$ and $j\in \{1,\dots,N-1\}$. Then, with the convention that $\operatorname{pr}^{2,N} = \operatorname{pr}^{2,0}$, using the permutation $i \mapsto i-1$, we have
		\begin{align*}
			0\geq & \sum_{i=1}^N \int_{(\X_1\times \X_2)^N} \cost(x_{1,i},x_{2,i}) - \cost(x_{1,i},x_{2,i-1}) d\theta 
			\\
			= &  N\int_{\X_1\times \X_2} \cost d\pi - (N-1)\int_{\X_1\times\X_2} \cost d\pi' - \int_{(\X_1\times \X_2)^N} \cost(x_{1,1},x_{2,N}) d\theta
			\\
			= & (N-1) \left( \int \cost d\pi - \C(\mu_1,\mu_2) \right)  + \left( \int \cost d\pi - \int_{(\X_1\times \X_2)^N} \cost(x_{1,1},x_{2,N}) d \theta \right)
			\\
			\geq & (N-1) \left( \int \cost d\pi - \C(\mu_1,\mu_2) \right)  + \left( \int \cost d\pi - \ttA_1(\mu_1) - \ttA_2(\mu_2) \right). 
		\end{align*}
		The second term on the right hand side is finite and does not depend on $N$. On the other hand, the factor $ \int \cost d\pi - \C(\mu_1,\mu_2)$ is non-negative, and passing to the limit as $N\to+\infty$, the previous computation gives that it must be null, yielding optimality of $\pi$.
		
		\textit{2.} Define $\mathfrak{F}:= (\operatorname{pr}^1_\sharp,\operatorname{pr}^2_\sharp) \rmF\subset \PP_{\tta_1}(\X_1)\times \PP_{\tta_2}(\X_2)$. Let $N\geq 1$ and, for $i \in \{1,\dots,N\}$, $(\mu_{1,i},\mu_{2,i}) = (\operatorname{pr}^1_\sharp, \operatorname{pr}^2_\sharp)\pi_i$ for some $\pi_i \in \rmF$. We build $\theta \in \PP\big((\X_1\times \X_2)^N\big)$ in the following way: for all $j\in\{2,\dots,N\}$, consider $\pi_j'\in \Gamma_{o,\cost}(\mu_{1,j},\mu_{2,j-1})$, and iterating the gluing lemma \cite[Lemma 5.3.2]{AGS08}, there exists a $\theta \in \PP\big((\X_1\times \X_2)^N\big)$ satisfying
		\[\big( \operatorname{pr}^{1,i}, \operatorname{pr}^{2,i} \big)_\sharp\theta = \pi_i \ \text{ for all }i\in\{1,\dots,N\}, \quad \big(\operatorname{pr}^{1,j}, \operatorname{pr}^{2,j-1} \big)_\sharp\theta = \pi_j' \ \text{ for all }j\in\{2,\dots,N\}.\]
		Then, thanks to the optimality of $\pi_i$ and $\pi_j'$, it holds (we denote $\mu_{2,0} = \mu_{2,N}$ and $x_{2,0} = x_{2,N}$)
		\begin{equation}\label{eq:opt}
			\begin{aligned}
				\sum_{i=1}^N & \C(\mu_{1,i},\mu_{2,i}) = \sum_{i=1}^N \int_{(\X_1\times \X_2)^N} \cost(x_{1,i},x_{2,i}) d \theta \leq \sum_{i=1}^N  \int_{(\X_1\times \X_2)^N} \cost(x_{1,i},x_{2,i-1}) d \theta \\
				= & \sum_{j=2}^N \C(\mu_{1,j},\mu_{2,j-1}) + \int_{(\X_1\times \X_2)^N} \cost(x_{1,1},x_{2,N}) d\theta \leq \sum_{j=2}^N \C(\mu_{1,j},\mu_{2,j-1}) + \ttA_1(\mu_{1,1}) + \ttA_2(\mu_{2,N}). 
			\end{aligned}
		\end{equation}
		
		Now, by contradiction assume that there exist $N\geq 1$ and $\big\{(\mu_{1,i},\mu_{2,i}) \ : \ i\in \{1,\dots,N\}\big\} \subset \mathfrak{F}$ such that 
		\begin{equation}\label{eq:eps}
			\varepsilon:= \sum_{i=1}^N \C(\mu_{1,i},\mu_{2,i}) - \sum_{i=1}^N  \C(\mu_{1,i},\mu_{2,i-1}) >0.
		\end{equation}
		Then, for all $\ell\geq 2$, consider $\big\{(\tilde{\mu}_{1,j}, \tilde{\mu}_{2,j}) \ : \ i\in \{ 1,\dots,\ell N\} \big\}$ defined as
		\[\tilde{\mu}_{1,j} = \mu_{1,i}, \quad \tilde{\mu}_{2,j} = \mu_{2,i} \qquad \text{ if }j = i + k N \ \text{for some }k\in\{0 ,\dots,\ell-1\}.\]
		With the previous convention, notice that $\tilde{\mu}_{2,0} = \tilde{\mu}_{2,\ell N} = \mu_{2,N}$. Combining \eqref{eq:opt} and \eqref{eq:eps} we reach a contradiction by sending $\ell \to+\infty$:
		\begin{align*}
			\ttA_{1} & (\mu_{1,1}) + \ttA_{2}(\mu_{2,N}) - \C(\mu_{1,1},\mu_{2,N}) \geq \sum_{j=1}^{\ell N} \C(\tilde{\mu}_{1,j}, \tilde{\mu}_{2,j}) - \C(\tilde{\mu}_{1,j}, \tilde{\mu}_{2,j-1}) 
			\\
			= &  \sum_{k=0}^{\ell-1} \sum_{i=1}^{N} \C({\mu}_{1,i+kN}, {\mu}_{2,i+kN}) -  \C({\mu}_{1,i+kN},{\mu}_{2,i+ kN-1}) = \sum_{k=0}^{\ell-1} \sum_{i=1}^{N} \C({\mu}_{1,i}, {\mu}_{2,i}) -  \C({\mu}_{1,i},{\mu}_{2,i-1}) = \ell \varepsilon
		\end{align*}
		
		\textit{3.} Consider $\mathfrak{F} \subset \PP_{\tta_1}(\X_1)\times \PP_{\tta_2}(\X_2)$ as in the previous point. By Theorem \ref{thm:RR}, there exists $\phi_1\colon\PP_{\tta_1}(\X_1) \to [-\infty,+\infty)$ $\C$-concave such that $\mathfrak{F} \subseteq \partial_{\C}^+\phi_1$. By the very definition of $\mathfrak{F}$, Proposition \ref{prop: char tot superdiff} and Claim \textit{1.} we conclude:
		\[
		\begin{aligned}
			\rmF \subseteq &  \left\{ \pi \in \PP_{\tta_1\oplus \tta_2}(\X_1\times \X_2) :  \pi \in \Gamma_{o,\cost}(\operatorname{pr}^1_\sharp \pi, \operatorname{pr}^2_\sharp \pi), \ (\operatorname{pr}^1_\sharp \pi, \operatorname{pr}^2_\sharp \pi) \in \mathfrak{F}  \right\}
			\\
			\subseteq &  
			\left\{ \pi \in \PP_{\tta_1\oplus \tta_2}(\X_1\times \X_2) :  \pi \in \Gamma_{o,\cost}(\operatorname{pr}^1_\sharp \pi, \operatorname{pr}^2_\sharp \pi), \ (\operatorname{pr}^1_\sharp \pi, \operatorname{pr}^2_\sharp \pi) \in \partial_{\C}^+\phi_1  \right\} = \partial_{\ttt,\C}^+\phi_1.
		\end{aligned}
		\]
		With the same argument, the last assertion follows as well.
	\end{proof}
	
	\begin{theorem}\label{thm: char opt RGamma}
		Let $\cost\colon \X_1\times \X_2 \to [0,+\infty]$ be lower semicontinuous and such that $\cost \leq \tta_1\oplus \tta_2$, for some $\tta_i \colon \X_i \to [0,+\infty]$ Borel measurable. Let $\ttM_i \in \PP(\PP(\X_i))$, $i=1,2$, be satisfying Assumption \ref{ass: integr} and denote $\C$ as in \eqref{eq:60}. 
		\\
		Let $\ttP \in \RGamma(\ttM_1,\ttM_2)$. The following are equivalent:
		\begin{enumerate}
			\item $\ttP\in \RGamma_{o,\C}(\ttM_1,\ttM_2)$;
			\item $\ttP$ is concentrated on a totally $\C$-cyclically monotone subset;
			\item there exists $\phi_1:\PP_{\tta_1}(\X_1)\to [-\infty,+\infty)$ $\C$-concave such that $\phi_1\in L^1(\ttM_1)$, $\phi^{\C}\in L^1(\ttM_2)$ and 
			\begin{equation}\label{eq:112}
				\phi_1\big( \operatorname{pr}^1_\sharp \pi\big) + \phi_1^{\C}\big(\operatorname{pr}^2_\sharp \pi \big) = \int_{\X_1 \times \X_2} \cost(x_1,x_2) d\pi(x_1,x_2) \quad \text{ for }\ttP\text{-a.e. }\pi,
			\end{equation}
			i.e. $\ttP$ is concentrated on $\partial_{\ttt,\C}^+\phi_1$.
		\end{enumerate}
	\end{theorem}
	
	\begin{proof}
		\textit{1.}$\implies$\textit{2.} 
		It is implied by putting together Lemma \ref{lemma: rand coupl}, Theorem \ref{thm:char opt plans} and \eqref{eq:111}.
		
		\textit{2.}$\implies$\textit{3.} It is a consequence of Proposition \ref{prop: char tot mon}, Proposition \ref{prop: char tot superdiff} and Theorem \ref{thm:char opt plans}.
		
		\textit{3.}$\implies$\textit{1.} Let $\Pi := (\operatorname{pr}^1_\sharp,\operatorname{pr}^2_\sharp)_\sharp\ttP$. The optimality of $\ttP$ follows: 
		\begin{align*}
			\hspace{3cm} \calC(\ttM_1,\ttM_2) \leq & \int_{\PP(\X_1\times \X_2)} \int_{\X_1 \times \X_2} \cost(x_1,x_2) d\pi(x_1,x_2) d\ttP(\pi) 
			\\
			= & \int_{\PP(\X_1)\times \PP(\X_2)} \phi_1(\mu_1) + \phi_1^{\C}(\mu_2) d\Pi(\mu_1,\mu_2)
			\leq \calC(\ttM_1,\ttM_2). \hspace{2.8cm}\qedhere
		\end{align*}
	\end{proof}

	The previous statements extends the results presented in \cite[Sections 4 and 5]{PS25convex}. Moreover, here we provided different proofs that didn't involve a  lifting procedure. In the next section, we provide such a lifting procedure that fits this setting as well.

	\section{Lions' representations for $\PP(\X)$ and the lifted OT problems}\label{sec: Lions}
	
	\subsection{Convergence in probability}
	Let $(\X,\tau)$ be a Polish space and $(\rmQ, \cF_{\rmQ}, \bbM)$ be an atomless standard Borel probability space, that is, there exists a topology $\tau_\rmQ$ such that $(\rmQ,\tau_\rmQ)$ is Polish, $\cF_{\rmQ}$ is the generated Borel $\sigma$-algebra, and $\bbM$ is an atomless Borel probability measure. Consider then 
	\begin{equation}
		\XX := L^0(\rmQ,\cF_{\rmQ}, \bbM; \X),
	\end{equation}
	that is the space of measurable random variables $\Z\colon\rmQ \to \X$ quotiented by the $\bbM$-a.e. equivalence relation.
	
	\begin{definition}
		Let $\d$ be any complete distance on $\X$ inducing the topology $\tau$. Given $\Z_n,\Z \in \XX$, we say 
		\begin{equation}\label{eq:80}
			\Z_n \overset{\bbM_\d}{\longrightarrow} \Z 
			\iff 
			\forall \varepsilon >0 \quad \bbM\big(\d(\Z_n,\Z) \geq \varepsilon \big) \to 0. 
		\end{equation}
	\end{definition}
	
	The next lemma shows that this convergence can be metrized by a distance $D_{\bbM,\d}$ which makes $\XX$ a complete and separable metric space. Moreover, the topology induced by it does not depend on the choice of $\d$, giving that the convergence \eqref{eq:80} is induced by a Polish topology, that is usually called topology of convergence in probability.
	
	\begin{lemma}\label{lemma:top prob}
		Let $\Z_n,\Z \in \XX$. The following are equivalent:
		\begin{enumerate}
			\item $\Z_n \overset{\bbM_\d}{\longrightarrow} \Z$;
			\item $D_{\bbM,\d}(\Z_n,\Z) := \int_{\rmQ} 1\wedge \d(\Z_n,\Z) d\bbM \to 0$;
			\item for all $K\subset \X$ compact and for all $U \subset \X\times \X$ open such that $\Delta_K:=\{(x,x) \ : \ x \in K\} \subset U$, it holds
			\[\bbM\big( \Z\in K, \ (\Z_n,\Z) \notin U \big) \to 0.\]
		\end{enumerate}
		Moreover, $D_{\bbM,\d}$ is a distance that makes $\XX$ a complete and separable space. 
	\end{lemma}
	
	\begin{proof}
		For the equivalence between \textit{1.} and \textit{2.} we refer to \cite[Lemma 5.2]{kallenberg1997foundations}. Regarding \textit{1.}$\implies$\textit{3.}, since $K$ is compact there exist $x_1,\dots,x_N \in K$ and $r_1,\dots,r_N >0$ such that $B(x_i,r_i) \times B(x_i,r_i) \subset U$, and $\{B(x_i, r_i/2)\}_{i\leq N}$ is a cover of $K$, where we denoted with $B(x,r)$ the ball of center $x$ and radius $r$ with respect to the distance $\d$. Let $\delta := \min r_i/2$. Then, for all $z\in K$, there exists $j\leq N$ such that $z \in B(x_j,r_j/2)$, so that if $\d(y,z)<\delta$, then $\d(y,x_j) < r_j$. Thus, $\Z \in K$ and $\d(\Z_n,\Z)<\delta$ imply $(\Z_n,\Z) \in U$, and in particular 
		\[\{\d(\Z_n,\Z) \geq \delta\} \supseteq  \{ \Z \in K, \ \d(\Z_n,\Z)\geq \delta\} \supseteq \{ \Z \in K, \ (\Z_n,\Z) \notin U\}. \]
		Evaluating them with respect to $\bbM$ and passing to the limit, we conclude.
		
		Vice versa, assume \textit{3}. Fix $\varepsilon>0$. For all $\eta>0$, there exists $K_\eta \subset X$ compact such that $\bbM(\Z \in K_\eta) \geq 1-\eta$. Moreover, there exist $x_1,\dots,x_N \in K_\eta$ such that $K_\eta$ is covered by $\{B(x_i, \varepsilon/2)\}_{i\leq N}$. Then, by assumption \[U_\eta:= \bigcup_{i\leq N} B(x_i,\varepsilon/2) \times B(x_i,\varepsilon/2) \implies \bbM(\Z \in K_\eta, \ (\Z_n,\Z) \notin U_\eta) \to 0.\]
		Moreover, if $z \in K_\eta$ and $\d(y,z)\geq\varepsilon$, then for all $i\leq N$ such that $\d(z,x_i) <\varepsilon/2$, we have $\d(y,x_i) \geq \d(y,z) - \d(z,x_i) >\varepsilon/2$, and in particular $(y,z) \notin U_\eta$. Then, for all $\varepsilon>0$ and for all $\eta>0$, we have 
		\[
		\begin{aligned}
			\limsup_{n\to+\infty} \bbM(\d(\Z_n,\Z)>\varepsilon) & \leq \bbM(\Z \notin K_\eta) + \limsup_{n\to +\infty} \bbM(\Z\in K_\eta, \d(\Z_n,Z) >\varepsilon) 
			\\
			& \leq \eta+ \limsup_{n\to+\infty} \bbM(\Z \in K_\eta, \ (\Z_n,\Z) \notin U_\eta) = \eta,
		\end{aligned}
		\]
		and the proof is concluded by arbitrariness of $\eta>0$.
		The last part follows because $(\XX,D_{\bbM,\d})$ coincides with $L^1(\rmQ,\cF_{\rmQ}, \bbM; \X)$ endowed with the $L^1$-metric induced by the distance $1\wedge \d$.
	\end{proof}
	
	\subsection{Lions' liftings}\label{subsec: lions}
	Let $\X_1$ and $\X_2$ be two Polish spaces and consider for $i=1,2$ the Polish spaces $\XX_i:= L^0(\rmQ, \cF_{\rmQ}, \bbM; \X_i)$ endowed with the topology of convergence in probability. Define the maps
	\begin{equation}\label{eq:81}
		\begin{aligned}
			\iota_1\colon \XX_1 & \to \PP(\X_1), \qquad \quad
			\iota_2\colon \XX_2  \to \PP(\X_2), \qquad \quad
			\iota_{1,2}\colon \XX_1\times \XX_2 \to \PP(\X_1\times \X_2), \qquad
			\\
			\Z_1 & \mapsto (\Z_1)_\sharp \bbM 
			\hspace{1.72cm} \Z_2 \mapsto (\Z_2)_\sharp \bbM
			\hspace{1.94cm} (\Z_1,\Z_2) \mapsto (\Z_1,\Z_2)_\sharp \bbM.
		\end{aligned}
	\end{equation}
	
	It is known that the previous maps are all surjective and continuous (see \cite{CSS2}). However, if we fix $\Z_1 \in \XX_1$, then the map
	\begin{equation}\label{eq: mp}
		\Z_2 \mapsto (\Z_1,\Z_2)_\sharp \bbM \in \PP(\X_1\times \X_2)\cap \{\pi : \operatorname{pr}^1_\sharp \pi = (\Z_1)_\sharp \bbM\}
	\end{equation}
	is not surjective in general. Indeed, if the $\sigma$-algebra generated by $\Z_1$ is the whole $\cF_{\rmQ}$, then all $\Z_2$ can be written as maps of $\Z_1$, giving that the image of $\iota_{1,2}(\Z_1,\cdot)$ is the set of deterministic couplings $\PPdet(\X_1\times \X_2)\cap \{\pi : \operatorname{pr}^1_\sharp \pi = (\Z_1)_\sharp \bbM\}$. However, if $\Z_1$ admits an independent randomization, i.e. there exists $U\colon \rmQ \to \rmQ$ independent from $\Z_1$ such that $U_\sharp \bbM = \bbM$, then the map in \eqref{eq: mp} is surjective. 
	
	Due to this last observation, we introduce a weaker lifting procedure: for $i=1,2$ we define
	\begin{equation}\label{eq:82}
		\XX_i^{\operatorname{IR}} := \left\{(\Z_i^{\operatorname{IR}},U_i) \in L^0(\rmQ,\mathcal{F}_{\rmQ}, \bbM; \X_i \times \rmQ) \ : \ \Z_i^{\operatorname{IR}} \perp\!\!\!\perp U_i, \ (U_i)_\sharp \bbM = \bbM \right\},
	\end{equation}
	together with the (surjective) maps
	\begin{equation}\label{eq:83}
		\begin{aligned}
			\iota_1^{\operatorname{IR}}\colon \XX_1^{\operatorname{IR}} & \to \PP(\X_1), \qquad \qquad
			\iota_2^{\operatorname{IR}}\colon \XX_2^{\operatorname{IR}}  \to \PP(\X_2), \qquad \qquad
			\iota_{1,2}^{\operatorname{IR}}\colon \XX_1^{\operatorname{IR}}\times \XX_2^{\operatorname{IR}} \to \PP(\X_1\times \X_2), \qquad
			\\
			(\Z_1^{\operatorname{IR}},U_1) & \mapsto (\Z_1^{\operatorname{IR}})_\sharp \bbM 
			\hspace{1.13cm} (\Z_2^{\operatorname{IR}},U_2) \mapsto (\Z_2^{\operatorname{IR}})_\sharp \bbM
			\hspace{0.88cm} (\Z_1^{\operatorname{IR}},U_1,\Z_2^{\operatorname{IR}},U_2) \mapsto (\Z_1^{\operatorname{IR}},\Z_2^{\operatorname{IR}})_\sharp \bbM.
		\end{aligned}
	\end{equation}

	The definition \eqref{eq:82} is justified by the next remark and lemma. Indeed, even if in \eqref{eq:83} the independent randomizations do not play any role, we need to keep track of them for topological reasons.
	
	\begin{remark}
		The space $\{\Z_i \in L^0(\rmQ,\mathcal{F}_\rmQ, \bbM; \X_i) : \Z_i \text{ admits an independent randomization}\}$ is not closed with respect to the topology of convergence in probability. Indeed, since all atomless standard Borel probability spaces are isomorphic, we can assume $\rmQ = [0,1]$ endowed with the Lebesgue measure. Then, define the real random variables
		\[\Z_n(q) := \frac{k}{2^n} \quad \text{ if } q\in [k2^{-n},(k+1)2^{-n}), \quad k\in\{0,\dots,2^{-n}-1\}.\]
		For all $n\in \N$, $\Z_n$ admits an independent randomization. In fact, the random variable
		\[U_n(q):= 2^{n}\big(q-k2^{-n}\big)\quad \text{ if } q\in [k2^{-n},(k+1)2^{-n}), \quad k\in\{0,\dots,2^{-n}-1\}\]
		pushes the Lebesgue measure into itself, and is independent of $\Z_n$. However, $\Z_n(q) \to q$ for all $q\in[0,1)$, and thus $Z_n$ converges in probability to the random variable $Z_\infty(q):= q$, which clearly cannot admit an independent randomization.
	\end{remark}

	\begin{lemma}
		For $i=1,2$, the space $\XX_{i}^{\operatorname{IR}}$ is closed with respect to the topology of convergence in probability in $ L^0(\rmQ,\mathcal{F}_{\rmQ}, \bbM; \X_i \times \rmQ)$.
	\end{lemma}
	
	\begin{proof}
		For simplicity, we avoid the use of the subscript $i$. Let $(\Z_n^{\operatorname{IR}},U_n) \overset{}{\to} (\Z,U)$. Independence of $\Z_n^{\operatorname{IR}}$ and $U_n$ implies that $\mathbb{E}[f(\Z_n^{\operatorname{IR}})g(U_n)] = \mathbb{E}[f(\Z_n^{\operatorname{IR}})] \mathbb{E}[g(U_n)]$ for all $f\in C_b(\X)$ and $g\in C_b(\rmQ)$. Since convergence in probability yields convergence in distribution, passing to the limit we immediately have $\mathbb{E}[f(\Z^{\operatorname{IR}})g(U)] = \mathbb{E}[f(\Z^{\operatorname{IR}})] \mathbb{E}[g(U)]$ for all test functions, which is equivalent to independence.
		Moreover, since all the $U_n$ satisfies $(U_n)_\sharp \bbM = \bbM$ by definition, then also $U_\sharp \bbM = \bbM$.
	\end{proof}

	Consider a lower semicontinuous function $\cost\colon \X_1\times \X_2 \to [0,+\infty]$, and recall the definition of $\C \colon \PP(\X_1) \times \PP(\X_2) \to [0,+\infty]$ from \eqref{eq:60}. Define 
	\begin{equation}\label{eq:91}
		\begin{aligned}
			& \hat{\C}\colon\XX_1\times \XX_2 \to [0,+\infty], \qquad \quad \ \  \hat{\C}(\Z_1,\Z_2) := \int_{\rmQ} \cost(\Z_1,\Z_2) d\bbM,
			\\
			& \hat{\C}^{\operatorname{IR}}\colon \XX_1^{\operatorname{IR}}\times \XX_2^{\operatorname{IR}} \to [0,+\infty], \qquad \hat{\C}^{\operatorname{IR}}(\Z_1^{\operatorname{IR}},U_1,\Z_2^{\operatorname{IR}},U_2) := \int_{\rmQ} \cost(\Z_1^{\operatorname{IR}},\Z_2^{\operatorname{IR}}) d\bbM.
		\end{aligned}
	\end{equation}

	\begin{lemma}\label{lemma:lusin std Borel}
		Let $(\rmQ,\cF_{\rmQ},\bbM)$ be an atomless standard Borel probability space. Then, for all $B\in \cF_{\rmQ}$ with $\bbM(B)>0$, the triplet $(B,\cF_{\rmQ}|_B, \frac{1}{\bbM(B)}\bbM\mres B)$ is an atomless standard Borel probability space.
	\end{lemma}
	
	\begin{proof}
		The non-atomicity of the reference measure is trivial. On the other hand, thanks to \cite[Theorem 2, pp. 95]{Schwartz73}, $B$ endowed with the subspace topology $\tau|_B$ is a Lusin space, that is there exists $\tau'\supset \tau|_B$ topology on $B$ that makes it Polish. On the other hand, \cite[Corollary 2, pp. 101]{Schwartz73} gives that the $\sigma$-algebra generated by $\tau'$ coincides with the one generated by $\tau|_B$, which coincides with $\cF_{\rmQ}|_B = \{A \cap B \ : \ A\in \cF_{\rmQ}\}$. 
	\end{proof}
	
	\begin{lemma}\label{lemma:key}
		Assume that $\cost$ is lower semicontinuous.
		\begin{enumerate}
			\item $\C(\iota_1(\Z_1),\iota_2(\Z_2)) \leq \hat{\C}(\Z_1,\Z_2)$ and equality holds if and only if $\iota_{1,2}(\Z_1,\Z_2) \in \Gamma_{o,\cost}(\iota_1(\Z_1),\iota_2(\Z_2))$. Similarly, $$\C(\iota_1^{\operatorname{IR}}(\Z_1^{\operatorname{IR}},U_1),\iota_2^{\operatorname{IR}}(\Z_2^{\operatorname{IR}},U_2)) \leq \hat{\C}^{\operatorname{IR}}(\Z_1^{\operatorname{IR}},U_1,\Z_2^{\operatorname{IR}},U_2)$$ and equality holds if and only if $\iota_{1,2}^{\operatorname{IR}}(\Z_1^{\operatorname{IR}},U_1,\Z_2^{\operatorname{IR}},U_2) \in \Gamma_{o,\cost}(\iota_1^{\operatorname{IR}}(\Z_1^{\operatorname{IR}},U_1),\iota_2^{\operatorname{IR}}(\Z_2^{\operatorname{IR}},U_2))$.
			\item $\hat{\C}$ and $\hat{\C}^{\operatorname{IR}}$ are lower semicontinuous.
			\item For $\mu_1 \in \PP(\X_1)$, $\mu_2 \in \PP(\X_2)$ and $ (\Z_1^{\operatorname{IR}},U_1) \in (\iota^{\operatorname{IR}}_1)^{-1}(\mu_1)$, there exists $(\Z_2^{\operatorname{IR}},U_2) \in (\iota_2^{\operatorname{IR}})^{-1}(\mu_2)$ such that 
			\begin{equation}\label{eq:92}
				\C(\mu_1,\mu_2) = \hat{\C}^{\operatorname{IR}}(\Z_1^{\operatorname{IR}},U_1,\Z_2^{\operatorname{IR}},U_2).
			\end{equation}
			\item Assume that $\cost$ is continuous (still possibly taking the value $+\infty$). For all $\mu_1 \in \PP(\X_1)$, $\mu_2 \in \PP(\X_2)$ and $\Z_1 \in \iota_1^{-1}(\mu_1)$, it holds
			\begin{equation}\label{eq:93}
				\C(\mu_1,\mu_2) = \inf_{\Z_2 \in \iota_2^{-1}(\mu_2)} \hat{\C}(\Z_1,\Z_2),
			\end{equation}
			possibly infinitely valued. Moreover, if $\mu_1$ is atomless we could restrict the infimum to the $\Z_2 \in \iota_2^{-1}(\mu_2)$ for which $(\Z_1,\Z_2)_\sharp \bbM$ is deterministic, or equivalently $\Z_2$ is $\sigma(\Z_1)$-measurable.
		\end{enumerate}
	\end{lemma}
	
	\begin{proof}
		\begin{enumerate}
			\item It is straightforward, since $\iota_{1,2}(\Z_1,\Z_2) \in \Gamma(\iota_1(\Z_1), \iota_2(\Z_2))$. Similarly for $\hat{\C}^{\operatorname{IR}}$.
			\item 
			Consider $\Z_{i,n}, \Z_i \in \XX_i$ such that $\Z_{i,n} \to \Z_i$ in $\bbM$-probability, for $i=1,2$. Without loss of generality assume that $\lim_{n\to+\infty}\hat{C}(\Z_{1,n},\Z_{2,n})$ exists and is finite. Consider a subsequence $\{n_k\}_{k\in \N}$ such that $\Z_{i,n_k} \to \Z_i$ $\bbM$-a.s. Then, by Fatou's lemma
			\[\begin{aligned}
				\hat{\C}(\Z_1,\Z_2) & \leq \int_{\rmQ} \liminf_{k\to+\infty} \cost(\Z_{1,n_k},\Z_{2,n_k}) d\bbM 
				\\
				& \leq \liminf_{k\to +\infty} \int_{\rmQ} \cost(\Z_{1,n_k},\Z_{2,n_k}) d\bbM = \lim_{n\to+\infty} \hat{C}(\Z_{1,n},\Z_{2,n}).
			\end{aligned}\]
			A similar argument yields lower semicontinuity of $\hat{\C}^{\operatorname{IR}}$.
			\item Let $(\Z_1^{\operatorname{IR}},U_1)\in\XX_{1}^{\operatorname{IR}}$. 
			Consider $\pi \in \Gamma_{o,\cost}(\mu_1,\mu_2)$ and let $\{\pi_{x_1}\}_{x_1\in \X_1}\subset \PP(\X_2)$ be its disintegration with respect to the first marginal $\mu_1$. 
			Then, there exists  $F \colon \X_1\times \rmQ \to \X_2\times \rmQ$ Borel measurable such that $F(x_1,U_1)_\sharp \bbM = \pi_{x_1}\otimes \bbM $ for $\mu_1$-a.e. $x_1\in \X_1$ (see \cite[Lemma 3.2]{kallenberg1997foundations}). Then, define 
			\((\Z_2^{ \operatorname{IR}}, U_2) := F(Z_1^{\operatorname{IR}},U_1).\) By construction, its joint law is $\mu_2\otimes \bbM$ and $\big(\Z_1^{\operatorname{IR}},\Z_2^{\operatorname{IR}}\big)_\sharp \bbM = \pi$, from which \eqref{eq:92} follows. 
			\item The atomless case is a consequence of \cite[Theorem B]{pratelli2007equality}, which yields also the last part of the statement. Also, if $\C(\mu_1,\mu_2) = +\infty$, then \eqref{eq:93} is trivial. For the general case, assume $\C(\mu_1,\mu_2)<+\infty$ and write 
			\[\mu_1 = \mu_1^{\operatorname{diff}} + \mu_1^{\operatorname{at}}, \quad \mu_1^{\operatorname{at}} = \sum_{i\in I} a_i \delta_{x^i_1},\]
			for some countable index set $I$, $\{x^i_1\}_{i\in I}\subset X_1$, $\{a_i\}_{i\in I}\subset (0,1]$ with $\sum_{i\in I} a_i \in (0,1]$. Then 
			\[\rmQ^i := \Z_1^{-1}(\{x_1^i\}), \quad \overline{\rmQ} := \rmQ \setminus \bigcup_{i\in I} \rmQ^i \quad \implies \quad \bbM(\rmQ^i) = a_i, \quad \bbM(\overline{\rmQ}) = 1 - \sum_{i\in I} a_i.\]
			Consider now $\pi \in \Gamma_{o,\cost}(\mu_1,\mu_2)$ and its disintegration $\{\pi_{x_1}\} \subset \PP(\X_2)$ with respect to $\mu_1$, giving 
			\[\pi(A\times B) = \int_{A} \pi_{x_1}(B) d\mu_1(x_1) = \int_{A} \pi_{x_1}(B) d\mu_1^{\operatorname{diff}}(x_1) + \int_{A} \pi_{x_1}(B) d\mu_1^{\operatorname{at}}(x_1) =: \pi^{\operatorname{diff}}(A\times B) + \pi^{\operatorname{at}}(A\times B).\]
			Thanks to Lemma \ref{lemma:lusin std Borel}, $(\rmQ^i, \cF_{\rmQ}|_{\rmQ^i}, \frac{1}{a_i} \bbM \mres{\rmQ^i})$ is a standard Borel space, thus for all $i\in I$ there exists $\Z_2^i : \rmQ^i \to \X_2$ measurable such that $(\Z_2^i)_\sharp \frac{1}{a_i}\bbM \mres {\rmQ^i} = \pi_{x_1^i}$. 
			Now we have two cases:
			\begin{itemize}
				\item $\mu_1^{\operatorname{diff}}\equiv 0$. Fix $x_2 \in \X_2$, and define the measurable map
				\[\Z_2 (q) = \begin{cases}
					\Z_2^{i}(q) & \text{ if }q\in \rmQ^i
					\\
					x_2 & \text{ if }q \in \overline{\rmQ}.
				\end{cases}\]
				The next computation gives that $(\Z_1,\Z_2)_\sharp \bbM = \pi$: fix $g:\X_1\times \X_2 \to [0,1]$ measurable
				\begin{align*}
					\int_{\rmQ} g(\Z_1,\Z_2) d\bbM 
					= &
					\sum_{i\in I} a_i \int_{\rmQ^i} g(x_1^i,\Z_2^i) d\frac{1}{a_i}\bbM\mres{\rmQ^i} 
					=  \sum_{i\in I} a_i \int_{\X_2} g(x_1^i,x_2) d\pi_{x_1^i}(x_2) 
					\\
					= & \int_{\X_1} \int_{\X_2} g(x_1,x_2) d\pi_{x_1}(x_2) d\mu_1(x_1)  
					=  \int_{\X_1\times \X_2} g(x_1,x_2) d\pi(x_1,x_2).
				\end{align*}
				\item $\bar{a}:=\mu_1^{\operatorname{diff}}(\X_1)>0$. We can split the measure $\mu_2$ in 
				\[\overline{\mu}_2 := \int_{\X_1} \pi_{x_1} d\mu_1^{\operatorname{diff}}(x_1), \quad \mu_2^i := a_i \pi_{x_1^i} \quad \implies \quad \mu_2 = \overline{\mu}_2 + \sum_{i\in I} \mu_2^i.\]
				It is clear that $\frac{1}{\bar{a}} \pi^{\operatorname{diff}}$ is $\C$-optimal between $\frac{1}{\bar{a}} \mu_1^{\operatorname{diff}}$ and $\frac{1}{\bar{a}}\overline{\mu}_2$, otherwise we could find a coupling between $\mu_1$ and $\mu_2$ with lower transport cost than the optimal one. Then, since the transport cost $\C(\frac{1}{\bar{a}} \mu_1^{\operatorname{diff}},\frac{1}{\bar{a}}\overline{\mu}_2)<+\infty$, applying again Lemma \ref{lemma:lusin std Borel} and \cite[Theorem B]{pratelli2007equality}, for all $\varepsilon>0$ there exists $T_\varepsilon:\X_1 \to \X_2$ such that $\big(T_\varepsilon \circ \Z_1|_{\overline{\rmQ}}\big)_\sharp \big(\frac{1}{\overline{a}}\bbM\mres{\overline{\rmQ}}\big) = \frac{1}{\overline{a}}\overline{\mu}_2$ and 
				\[\int_{\overline{\rmQ}} \cost(\Z_1, T_\varepsilon\circ \Z_1) d\frac{1}{\bar{a}}\bbM\mres{\overline{\rmQ}} \leq \C\left(\frac{1}{\bar{a}} \mu_1^{\operatorname{diff}},\frac{1}{\bar{a}}\overline{\mu}_2\right) + \frac{1}{\bar{a}} \varepsilon.\]
				Define then the measurable random variable
				\[\Z_{2, \varepsilon}(q) := \begin{cases}
					\Z_2^{i}(q) & \text{ if }q\in \rmQ^i
					\\
					T_\varepsilon \circ \Z_1(q) & \text{ if }q \in \overline{\rmQ},
				\end{cases}\]
				for which it can be easily proved that $(\Z_{2,\varepsilon})_\sharp\bbM = \mu_2$ and $\hat{\C}(\Z_1,\Z_{2,\varepsilon}) \leq \C(\mu_1,\mu_2) + \varepsilon$:
				\begin{align*}
					\hspace{0.5cm}\hat{\C}(\Z_1,\Z_{2,\varepsilon}) = & 
					\bar{a}\int_{\overline{\rmQ}} \cost(\Z_1,T_\varepsilon\circ \Z_1) d \frac{1}{\bar{a}} \bbM + \sum_{i\in I} a_i\int_{\rmQ^i} \cost(x_1^i, \Z_2^i) d\frac{1}{a_i}\bbM 
					\\
					\leq &
					\bar{a} \left(\C(\frac{1}{\bar{a}} \mu_1^{\operatorname{diff}}, \frac{1}{\bar{a}} \overline{\mu}_2) + \frac{1}{\bar{a}}\varepsilon\right) + \sum_{i\in I} a_i \int_{Q^i} \cost(x_1^i, x_2) d\pi_{x_1^i}(x_2)
					\\
					= & \int_{\X_1\times \X_2} \cost(x_1,x_2) d\pi^{\operatorname{diff}}(x_1,x_2) + \int_{\X_1\times \X_2} \cost(x_1,x_2) d\pi^{\operatorname{at}}(x_1,x_2) + \varepsilon
					\\
					= & \C(\mu_1,\mu_2) + \varepsilon. \hspace{9.4cm}  \qedhere
				\end{align*}
			\end{itemize}
		\end{enumerate}
	\end{proof}

	It is worth observing that the last claim in the previous lemma may fail if $\cost$ is assumed to be only lower semicontinuous, as the following counterexample shows. Consider $\X_1 = \X_2 = \R^2$ and 
	\[\cost\big((x_1,y_1), (x_2,y_2)\big) = \mathds{1}_{\{y_1 \neq y_2\}}, \quad \mu_1:= \cH^1|_{\{0\}\times[0,1]}, \quad \mu_2 = \frac{1}{2}\cH^1|_{\{-1\}\times[0,1]} + \frac{1}{2}\cH^1|_{\{1\}\times[0,1]}. \] Clearly, $\cost$ is lower semicontinuous, since $\{(x_1,y_1,x_2,y_2)  :  y_1 = y_2\} \subset \R^4$ is closed, and $\C(\mu_1,\mu_2) = 0$ considering the coupling $\pi := \int_0^1 \delta_{(0,x)}\otimes \big(\frac{1}{2}\delta_{(-1,x)} + \frac{1}{2} \delta_{(1,x)} \big) dx$. 
	\\
	Then, consider $(\rmQ,\cF_{\rmQ}, \bbM) = ([0,1], \cB, \lambda)$ the standard unit interval with the Borel $\sigma$-algebra and the Lebesgue measure, and define $\Z_1(q) := (0,q)$. As observed above, $\Z_1$ generates the whole $\sigma$-algebra $\cB$, and in particular any $\Z_2$ such that $(\Z_2)_\sharp \lambda = \mu_2$ is of the form $T\circ \Z_1$ for some measurable $T: \R^2 \to \R^2$ such that $T_\sharp \mu_1 = \mu_2$. On the other hand, for all such $T's$, saying that $T=(T_1,T_2)$, it holds $T_2(0,q) \neq q$ for $\lambda$-a.e. $q\in [0,1]$, giving that $\hat{\C}(\Z_1,\Z_2)\geq 1$ for all $\Z_2$ such that $(\Z_2)_\sharp \lambda = \mu_2$. Indeed, by contradiction, assume there exists $A\subset [0,1]$ with $\lambda(A)>0$ such that $T_2(0,q) = q$ for all $q\in A$. Without loss of generality, we can assume that $T_1(0,q) = 1$ for all $q\in A$, up to substitute it with one of the two sets $A^{\pm} := \{q\in A \ : \ T_1(0,q) = \pm 1\}$. Then
	\[\frac{\lambda(A)}{2} = \mu_2(\{1\}\times A) = \mu_1(T^{-1}(\{1\}\times A))\geq \mu_1(\{0\}\times A) = \lambda(A),\]
	since $\{0\}\times A \subseteq T^{-1}(\{1\}\times A)$, thus we reached a contradiction. 
	
	Combining Lemma \ref{lemma:key} and Lemma \ref{lemma:top prob}, we obtain the following.
	
	\begin{corollary}
		Let $(\X,\tau)$ be a Polish space and $\d$ a distance inducing $\tau$ that makes $X$ a complete metric space. Then, the map 
		\[\iota\colon L^0(\rmQ, \cF_{\rmQ}, \bbM ; \X) \to \PP(\X), \quad \iota(\Z) = \Z_\sharp \bbM \]
		is $1$-Lipschitz, endowing the domain and codomain, respectively, with $D_{\bbM,\d}$ and $\sfw_{1,\d\wedge 1}$. Moreover, for all $\mu_1,\mu_2 \in \PP(\X)$ and for all $\Z_1 \in \iota^{-1}(\mu_1)$, there exists a sequence $\{\Z_{2,n}\}\subset \iota^{-1}(\mu_2)$ such that 
		\[D_{\bbM,\d}(\Z_1,\Z_{2,n}) \to \sfw_{1,\d\wedge 1}(\mu_1,\mu_2) \quad \text{ as }n\to +\infty.\]
	\end{corollary}
	
	A similar result holds for the indpendent randomization case. For a fixed Polish space $\X$, we define $\XX^{\operatorname{IR}}$ in the obvious way.
	
	\begin{corollary}
		Let $(\X,\tau)$ be a Polish space and $\d$ a distance inducing the product topology of $\X\times \rmQ$ and that makes it a complete metric space. Then, the map 
		\[\iota^{\operatorname{IR}}\colon \XX^{\operatorname{IR}} \to \PP(\X), \quad \iota(\Z^{\operatorname{IR}},U) = \Z^{\operatorname{IR}}_\sharp \bbM \]
		is $1$-Lipschitz, endowing the domain and codomain, respectively, with $D_{\bbM,\d\oplus \d_\rmQ}$ and $\sfw_{1,\d\wedge 1}$. Moreover, for all $\mu_1,\mu_2 \in \PP(\X)$ and for all $(\Z_1^{\operatorname{IR}},U_1) \in (\iota^{\operatorname{IR}})^{-1}(\mu_1)$, there exists $(\Z_{2}^{\operatorname{IR}},U_2) \in (\iota^{\operatorname{IR}})^{-1}(\mu_2)$ such that 
		\[D_{\bbM,\d\oplus \d_\rmQ}\big((\Z_1^{\operatorname{IR}},U_1),(\Z_{2}^{\operatorname{IR}},U_2)\big) = \sfw_{1,d\wedge 1}(\mu_1,\mu_2).\]
	\end{corollary}

	\subsection{Continuous cost: \texorpdfstring{$\C$}{}-concavity and \texorpdfstring{$\hat{\C}$}{}-concavity}\label{subsec: lions cont}

	In this section, we relate the OT problems \eqref{eq:61} and \eqref{eq:72}, to the OT problem associated to the cost $\hat{\C}$, with an underlying \textit{continuous} function $\cost:\X_1\times \X_2 \to [0,+\infty]$ that satisfies \eqref{eq:100}.

	We define the Borel set
	\begin{equation}\label{eq:repr cont}
		\XX_{i,\tta_i} := \left\{\Z_i \in \XX_i \ : \ \int_{\rmQ} \tta_i(\Z_i) d\bbM <+\infty\right\}.
	\end{equation}

	\noindent Notice that $\Z_i \in \XX_{i,\tta_i}$ if and only if $\iota_i\in \PP_{\tta_i}(\X_i)$, for $i=1,2$.	
	Thanks to the lifting procedure, we can also find a natural correspondence between Kantorovich potentials. 
	For all $\phi_i\colon\PP_{\tta_i}(\X_i) \to [-\infty,+\infty]$, define its lifting $\hat{\phi}_i\colon\XX_{i,\tta_i}\to [-\infty,+\infty]$ as 
	\begin{equation}\label{eq:120}
		\hat{\phi}_i(\Z_i):= \phi_i(\iota_i(\Z_i)) = \phi_i\big((\Z_i)_\sharp \bbM \big).
	\end{equation}
	Moreover, for all $\rmF\subset \PP(\X_1,\times \X_2)$ we define
	\[\hat{\rmF}:= \left\{ (\Z_1,\Z_2) \ : \ \iota_{1,2}(\Z_1,\Z_2) \in \rmF \right\} \subseteq \XX_1\times \XX_2.\]
	Note that $\hat{\phi}_i$ (resp. $\rmF$) is \textit{law-invariant}, in the sense that if $(\Z_i)_\sharp\bbM = (\W_i)_\sharp \bbM$, then $\hat{\phi}_i(\Z_i) = \hat{\phi}_i(\W_i)$ (resp. if $(\Z_1,\Z_2)_\sharp\bbM = (\W_1,\W_2)_\sharp\bbM$, then $(\Z_1,\Z_2)\in \rmF$ if and only if $(\W_1,\W_2) \in \rmF$). Moreover, if $\phi_1\colon\PP_{\tta_1}(\X_1) \to [-\infty, +\infty)$ and $\rmF \subset \PP_{\tta_1\oplus\tta_2}(\X_1\times \X_2)$, then 
	\begin{itemize}
		\item $\phi_1^{\C}\colon \PP_{\tta_2}(\X_2) \to [-\infty,+\infty)$, $\hat{\phi}_1\colon\XX_{1,\tta_1}\to [-\infty,+\infty)$ and $\hat{\phi}_1^{\hat{\C}}\colon \XX_{2,\tta_2}\to [-\infty,+\infty)$;
		\item $\hat{\rmF}\subset \XX_{1,\tta_1}\times \XX_{2,\tta_2}$.
	\end{itemize}
	
	\begin{proposition}\label{prop: lift cont}
		Let $\phi_1:\PP_{\tta_1}(\X_1) \to [-\infty, +\infty)$ be not identically $-\infty$. Then 
		\begin{enumerate}
			\item $\widehat{\phi_1^{\C}} = \hat{\phi}_1^{\hat{\C}}$;
			\item $\phi_1$ is $\C$-concave if and only if $\hat{\phi}_1$ is $\hat{\C}$-concave;
			\item $(\Z_1,\Z_2)\in \partial_{\hat{\C}}^+ \hat{\phi}_1$ if and only if $\big(\iota_1(\Z_1),\iota_2 (\Z_2)\big) \in \partial_{\C}^+\phi_1$ and $\hat{\C}(\Z_1,\Z_2) = \C\big(\iota_1(\Z_1), \iota_2(\Z_2)\big)$;
			\item $\rmF \subset \PP_{\tta_1\oplus\tta_2}(\X_1\times \X_2)$ is totally $\C$-cyclically monotone if and only if $\hat{\rmF}$ is $\hat{\C}$-cyclically monotone;
			\item $\widehat{\partial_{\ttt,\C}^+\phi_1} = \partial_{\hat{\C}}^+ \hat{\phi}_1$, that is $(\Z_1,\Z_2) \in \partial_{\hat{\C}}^+ \hat{\phi}_1$ if and only if $(\Z_1,\Z_2)_\sharp \bbM \in \partial_{\ttt,\C}^+\phi_1$.
		\end{enumerate}
	\end{proposition}
	
	\begin{proof}
		\textit{1.} Thanks to Lemma \ref{lemma:key}, Claim 4., for all $\Z_2 \in \XX_{2,\tta_2}$ we have
		\begin{align*}
			\widehat{\phi_1^{\C}}(\Z_2) = & \phi_{1}^{\C}\big(\iota_2(\Z_2)\big) = \inf_{\mu_1\in \PP_{\tta_1}(\X_1)}\C(\mu_1,\iota_2(\Z_2)) - \phi_1(\mu_1) 
			\\
			= & \inf_{\mu_1\in \PP_{\tta_1}(\X_1)} \inf_{\Z_1\in \iota_{1}^{-1}(\mu_1)}\hat{\C}(\Z_1,\Z_2) - \hat{\phi}_1(\Z_1) = \inf_{\Z_1\in \XX_{1,\tta_1}}\hat{\C}(\Z_1,\Z_2) - \hat{\phi}_1(\Z_1) = \hat{\phi}_1^{\hat{\C}}(\Z_2).
		\end{align*}
		
		\textit{2.} $\phi_1$ is $\C$-concave if and only if $\phi_1^{\C\C}=\phi_1$, if and only if $\widehat{\phi_1^{\C\C}} = \hat{\phi}_1$. Applying twice the previous point, $\widehat{\phi_1^{\C\C}} = {\hat{\phi}_1^{\hat{\C}\hat{\C}}}$, and it is equal to $\hat{\phi}_1$ if and only if $\hat{\phi}_1$ is $\hat{\C}$-concave.
		
		\textit{3.} Let $(\Z_1,\Z_2)\in \XX_{1,\tta_1}\times \XX_{2,\tta_2}$ such that $\big(\iota_1(\Z_1),\iota_2 (\Z_2)\big) \in \partial_{\C}^+\phi_1$ and $\hat{\C}(\Z_1,\Z_2) = \C\big(\iota_1(\Z_1), \iota_2(\Z_2)\big)$. Then, for all $\Z_1' \in \XX_{1,\tta_1}$ it holds
		\[\begin{aligned}
			\hat{\C}(\Z_1,\Z_2) - \hat{\phi}_1(\Z_1) & = \C\big(\iota_1(\Z_1),\iota_2(\Z_2)\big) - \phi_1(\iota_1(\Z_1)) 
			\\
			&\leq \C\big(\iota_1(\Z_1'), \iota_2(\Z_2)) - \phi_{1}(\iota_1(\Z_1')) \leq \hat{\C}(\Z_1',\Z_2) - \hat{\phi}_1(\Z_1').
		\end{aligned}\]
		Vice versa, assume that $(\Z_1,\Z_2)\in \partial_{\hat{\C}}^+ \hat{\phi}_1$. Then, applying Lemma \ref{lemma:key}, for all $\mu_1'\in \PP_{\tta_1}(\X_1)$ we have
		\begin{align*}
			\C\big( \iota_1(\Z_1),\iota_2(\Z_2) \big) & -  \phi_1(\iota_1(\Z_1)) \leq \hat{\C}(\Z_1,\Z_2) - \hat{\phi_1}(\Z_1) 
			\\
			\leq &\inf_{\Z_1'\in \iota_1^{-1}(\mu_1')} \hat{\C}(\Z_1',\Z_2) - \hat{\phi}_1(\Z_1') = \C(\mu_1',\iota_2(\Z_2)) - \phi_1(\mu_1').
		\end{align*}
		Finally, the previous computation with $\mu_1'=\iota_1(\Z_1)$ yields that all the inequalities are actually equalities, and in particular $\C\big(\iota_1(\X_1,\iota_2(\X_2))\big) = \hat{\C}(\Z_1,\Z_2)$.
		
		\textit{4.} Assume $\rmF$ is total $\C$-cyclically monotone. Let $N\geq 1$, $\sigma$ permutation of $\{1,\dots,N\}$ and $(\Z_{1,i},\Z_{2,i}) \in \hat{\rmF}$. Then, defining $\theta := \big((\Z_{1,1},\Z_{2,1}),\dots, (\Z_{1,N},\Z_{2,N})\big)_\sharp \bbM$, we have 
		\begin{equation}\label{eq:ccc}
			\sum_{i=1}^N \hat{\C}(\Z_{1,i},\Z_{2,i}) - \hat{\C}(\Z_{1,i},\Z_{2,\sigma(i)}) = \int_{(\X_1\times \X_2)^N} \sum_{i=1}^N \cost(x_{1,i},x_{2,i}) - \cost(x_{1,i}, x_{2,\sigma(i)}) d\theta \leq 0.
		\end{equation}
		
		Vice versa, if $\hat{\rmF}$ is $\hat{\C}$-cyclically monotone, consider $N\geq 1$, $\sigma$ permutation of $\{1,\dots,N\}$ and  $\theta \in \PP((\X_1\times \X_2)^N)$ such that $(\operatorname{pr}^{1,i},\operatorname{pr}^{2,i})_\sharp \theta \in \rmF$. Then, there exists $(\Z_{1,i},\Z_{2,i}) \in \XX_1\times \XX_2$ such that $\theta = \big((\Z_{1,1},\Z_{2,1}),\dots, (\Z_{1,N},\Z_{2,N})\big)_\sharp \bbM$. By definition of $\hat{\rmF}$, we have that $(\Z_{1,i},\Z_{2,i}) \in \hat{\rmF}$, and then \eqref{eq:ccc} yields that $\rmF$ is totally $\C$-cyclically monotone.

		\textit{5.} Let $(\Z_1,\Z_2) \in \partial_{\hat{\C}}^+ \hat{\phi}_1$. Consider $\theta \in \PP(\X_1\times \X_2\times \X_{1})$ such that the projection on the first two marginals gives the law of $(\Z_1,\Z_2)$.
		Point \textit{3.} implies
		\[\int_{\X_1\times\X_2\times\X_1} \cost(x_1,x_2) - \cost(x_1',x_2) d\theta(x_1,x_2,x_1') 
		\leq \hat{\C}(\Z_1,\Z_2) - \C(\operatorname{pr}^3_\sharp\theta, \iota_2(\Z_2)).\]
		Moreover, by Lemma \ref{lemma:key} and definition of superdifferential, we have
		\[\hat{\C}(\Z_1,\Z_2) - \hat{\phi}_1(\Z_1) \leq \inf_{\Z_1'\in \iota_1^{-1}(\operatorname{pr}^3_\sharp\theta)} \hat{\C}(\Z_1',\Z_2) - \hat{\phi}_1(\Z_1') = \C(\operatorname{pr}^3_\sharp\theta, \iota_2(\Z_2)) - \phi_1(\operatorname{pr}^3_\sharp\theta). \]
		Putting everything together, we obtain that $(\Z_1,\Z_2)_\sharp \bbM \in \partial_{\ttt,\C}^+\phi_1$.
		
		Vice versa, let $(\Z_1,\Z_2) \in \XX_{1,\tta_1}\times \XX_{2,\tta_2}$ be such that their joint law belongs to the total $\C$-superdifferential of $\phi_1$. Then, for all $\Z_1'\in \XX_{1,\tta_1}$, considering $\theta:= (\Z_1,\Z_2,\Z_1')_\sharp\bbM$, the definition of total $\C$-superdifferential immediately yields the conclusion.
	\end{proof}

	Let $\ttM_i\in \PP(\PP(\X_i))$ for $i=1,2$. Clearly, the maps $(\iota_i)_\sharp: \PP(\XX_i) \to \PP(\PP(\X_i))$ are surjective, and it is natural to relate the cost $\calC(\ttM_1,\ttM_2)$ with
	\begin{equation}
		\hat{\calC}(\ttm_1,\ttm_2) := \min \left\{ \int_{\XX_1\times \XX_2} \hat{\C}(\Z_1,\Z_2) d\mathfrak{P}(\Z_1,\Z_2) \ : \ \mathfrak{P}\in \Gamma(\ttm_1.\ttm_2) \right\},
	\end{equation}
	for $\ttm_i \in \PP(\XX_i)$ such that $(\iota_i)_\sharp \ttm_i = \ttM_i$.
	
	\begin{theorem}\label{thm: OT lift}
		Let $\ttM_i \in \PP(\PP(\X_i))$ be satisfying Assumption \ref{ass: integr}.
		\begin{enumerate}
			\item  If $\ttm_i \in \PP(\XX_i)$ is such that $(\iota_i)_\sharp \ttm_i = \ttM_i$, then $\ttm_i$ is concentrated on $\XX_{i,\tta_i}$ and 
			\begin{equation}\label{eq:130}
				\calC(\ttM_1,\ttM_2) \leq \hat{\calC}(\ttm_1,\ttm_2);
			\end{equation}
			\item There exist $\overline{\ttm}_i \in \PP(\XX_i)$ such that $(\iota_i)_\sharp \overline{\ttm}_i = \ttM_i$ and 
			\begin{equation}\label{eq:131}
				\calC(\ttM_1,\ttM_2) = \hat{\calC}(\overline{\ttm}_1,\overline{\ttm}_2);
			\end{equation}
			\item Let $\overline{\ttm}_i \in \PP(\XX_i)$ be satisfying $(\iota_i)_\sharp \overline{\ttm}_i = \ttM_i$ and \eqref{eq:131}. Then, $\mathfrak{P} \in \Gamma_{o,\hat{\C}}(\overline{\ttm}_1,\overline{\ttm}_2)$ implies $(\iota_{1,2})_\sharp \mathfrak{P}\in \RGamma_{o,\C}(\ttM_1,\ttM_2)$. Moreover, $(\phi_1,\phi_1^{\C})$ are optimal Kantorovich potentials for $\calC(\ttM_1,\ttM_2)$ if and only if $(\hat{\phi}_1,\hat{\phi}_1^{\hat{\C}})$ are optimal Kantorovich potentials for $\hat{\calC}(\overline{\ttm}_1,\overline{\ttm}_2)$.
		\end{enumerate}
	\end{theorem}
	
	\begin{proof}
		The first part easily follows directly from the definitions, while the second part is a consequence of the surjectivity of $(\iota_{1,2})_\sharp:\XX_1\times \XX_2 \to \PP(\X_1\times \X_2)$ (and thus the existence of a universally measurable right-inverse) and the existence of an optimal random coupling $\ttP \in \RGamma_{o,\C}(\ttM_1,\ttM_2)$. Regarding Claim \textit{3.}, the first part is trivial, while the second part follows from the previous Proposition: if $(\hat{\phi}_1,\hat{\phi}_1^{\hat{\C}})$ are optimal, then
		\begin{align*}
			\calC(\ttM_1,\ttM_2) = \hat{\calC}(\overline{\ttm}_1,\overline{\ttm}_2) = \int \hat{\phi}_1 d\overline{\ttm}_1 + \int \hat{\phi}_1^{\hat{\C}}d\overline{\ttm}_2 = \int \phi_1 d\ttM_1 + \int \phi_1^{\C} d\ttM_2 \leq \calC(\ttM_1,\ttM_2), 
		\end{align*}
		and similarly we have also the opposite implication.
	\end{proof}

	\subsection{Lower semicontinuous cost: \texorpdfstring{$\C$}{}-concavity and \texorpdfstring{$\hat{\C}^{\operatorname{IR}}$}{}-concavity}\label{subsec: lions IR}
	Here, we reproduce a similar scheme that allows us to deal with \textit{lower-semicontinuous} functions $\cost\colon\X_1,\times \X_2 \to [0,+\infty]$ that satisfies \eqref{eq:100}. The difference with the previous subsection is that the lack of the property \eqref{eq:93} makes necessary the use of the lifted spaces $\XX_i^{\operatorname{IR}}$. In particular, we define 
	\begin{equation}
		\begin{gathered}
			\XX_{i,\tta_i}^{\operatorname{IR}} := \left\{ (\Z_i^{\operatorname{IR}},U_i) \in \XX_i^{\operatorname{IR}} \ : \ \int_{\rmQ}\tta_i(\Z_i^{\operatorname{IR}}) d\bbM <+\infty \right\}, \\  \hat{\phi}_i^{\operatorname{IR}}(\Z_i^{\operatorname{IR}},U_i):= \phi_i(\iota_i^{\operatorname{IR}}(\Z_i^{\operatorname{IR}},U_i)) = \phi_i\big((\Z_i^{\operatorname{IR}})_\sharp \bbM \big),\smallskip
			\\ \hat{\rmF}^{\operatorname{IR}}:= \left\{ (\Z_1^{\operatorname{IR}},U_1,\Z_2^{\operatorname{IR}},U_2) \ : \ \iota_{1,2}^{\operatorname{IR}}(\Z_1^{\operatorname{IR}},U_1,\Z_2^{\operatorname{IR}},U_2) \in \rmF \right\} \subseteq \XX_1^{\operatorname{IR}}\times \XX_2^{\operatorname{IR}},
		\end{gathered}
	\end{equation}
	for given $\phi_i\colon\PP_{\tta_i}(\X_i) \to [-\infty,+\infty)$ and $\rmF \in \PP_{\tta_1\oplus \tta_2}(\X_1\times \X_2)$, which imply
	\begin{itemize}
		\item $\hat{\phi}_1^{\operatorname{IR}}\colon\XX_{1,\tta_1}^{\operatorname{IR}}\to [-\infty,+\infty)$ and $(\hat{\phi}_1^{\operatorname{IR}})^{\hat{\C}^{\operatorname{IR}}}: \XX_{2,\tta_2}^{\operatorname{IR}}\to [-\infty,+\infty)$;
		\item $\hat{\rmF}^{\operatorname{IR}}\subset  \XX_{1,\tta_1}^{\operatorname{IR}}\times \XX_{2,\tta_2}^{\operatorname{IR}}$.
	\end{itemize}
	
	\noindent Moreover, define
	\begin{equation}
		\hat{\calC}^{\operatorname{IR}}(\ttm_1^{\operatorname{IR}},\ttm_2^{\operatorname{IR}}) := \min \left\{ \int_{\XX_1^{\operatorname{IR}}\times \XX_2^{\operatorname{IR}}} \hat{\C}^{\operatorname{IR}}(\Z_1^{\operatorname{IR}},U_1,\Z_2^{\operatorname{IR}},U_2) d\mathfrak{P}^{\operatorname{IR}}(\Z_1^{\operatorname{IR}},\Z_2^{\operatorname{IR}}) \ : \ \mathfrak{P}^{\operatorname{IR}}\in \Gamma(\ttm_1^{\operatorname{IR}},\ttm_2^{\operatorname{IR}}) \right\},
	\end{equation}
	for $\ttm_i^{\operatorname{IR}} \in \PP(\XX_i^{\operatorname{IR}})$.
	
	Then, we can replicate the previous results, whose statements are exposed below. The proofs follows the exact same line of Proposition \ref{prop: lift cont} and Theorem \ref{thm: OT lift}, but exploiting Lemma \ref{lemma:key}, Claim \textit{3.}, instead of Claim \textit{4.}.
	
	\begin{proposition}
		Let $\phi_1\colon\PP_{\tta_1}(\X_1) \to [-\infty, +\infty)$ be not identically $-\infty$. Then 
		\begin{enumerate}
			\item $(\widehat{\phi_1^{\C}})^{\operatorname{IR}} = (\hat{\phi}_1^{\operatorname{IR}})^{\hat{\C}^{\operatorname{IR}}}$;
			\item $\phi_1$ is $\C$-concave if and only if $\hat{\phi}_1^{\operatorname{IR}}$ is $\hat{\C}^{\operatorname{IR}}$-concave;
			\item $(\Z_1^{\operatorname{IR}},U_1,\Z_2^{\operatorname{IR}},U_2)\in \partial_{\hat{\C}^{\operatorname{IR}}}^+ \hat{\phi}_1^{\operatorname{IR}}$ if and only if $\big(\iota_1^{\operatorname{IR}}(\Z_1,U_1),\iota_2^{\operatorname{IR}} (\Z_2^{\operatorname{IR}},U_2)\big) \in \partial_{\C}^+\phi_1$ and $\hat{\C}^{\operatorname{IR}}(\Z_1,U_1,\Z_2^{\operatorname{IR}},U_2) = \C\big(\iota_1^{\operatorname{IR}}(\Z_1^{\operatorname{IR}},U_1), \iota_2^{\operatorname{IR}}(\Z_2^{\operatorname{IR}},U_2)\big)$;
			\item $\rmF \subset \PP_{\tta_1\oplus\tta_2}(\X_1\times \X_2)$ is total $\C$-cyclically monotone if and only if $\hat{\rmF}^{\operatorname{IR}}$ is $\hat{\C}^{\operatorname{IR}}$-cyclically monotone;
			\item $(\widehat{\partial_{\ttt,\C}^+\phi_1})^{\operatorname{IR}} = \partial_{\hat{\C}^{\operatorname{IR}}}^+ \hat{\phi}_1^{\operatorname{IR}}$, i.e. $(\Z_1^{\operatorname{IR}},U_1,\Z_2^{\operatorname{IR}},U_2) \in \partial_{\hat{\C}^{\operatorname{IR}}}^+ \hat{\phi}_1^{\operatorname{IR}}$ if and only if $(\Z_1^{\operatorname{IR}},\Z_2^{\operatorname{IR}})_\sharp \bbM \in \partial_{\ttt,\C}^+\phi_1$.
		\end{enumerate}
	\end{proposition}
	
	\begin{theorem}\label{thm: OT lift IR}
		Let $\ttM_i \in \PP(\PP(\X_i))$ be satisfying Assumption \ref{ass: integr}.
		\begin{enumerate}
			\item  If $\ttm_i^{\operatorname{IR}} \in \PP(\XX_i^{\operatorname{IR}})$ is such that $(\iota_i^{\operatorname{IR}})_\sharp \ttm_i^{\operatorname{IR}} = \ttM_i$, then $\ttm_i^{\operatorname{IR}}$ is concentrated on $\XX_{i,\tta_i}^{\operatorname{IR}}$ and 
			\begin{equation}\label{eq:132}
				\calC(\ttM_1,\ttM_2) \leq \hat{\calC}^{\operatorname{IR}}(\ttm_1^{\operatorname{IR}},\ttm_2^{\operatorname{IR}});
			\end{equation}
			\item There exist $\overline{\ttm}_i^{\operatorname{IR}} \in \PP(\XX_i^{\operatorname{IR}})$ such that $(\iota_i^{\operatorname{IR}})_\sharp \overline{\ttm}_i^{\operatorname{IR}} = \ttM_i$ and 
			\begin{equation}\label{eq:133}
				\calC(\ttM_1,\ttM_2) = \hat{\calC}^{\operatorname{IR}}(\overline{\ttm}_1^{\operatorname{IR}}, \overline{\ttm}_2^{\operatorname{IR}});
			\end{equation}
			\item Assume $\overline{\ttm}_i^{\operatorname{IR}} \in \PP(\XX_i)$ such that $(\iota_i^{\operatorname{IR}})_\sharp \overline{\ttm}_i^{\operatorname{IR}} = \ttM_i$ and \eqref{eq:133} holds. Then, $\mathfrak{P}^{\operatorname{IR}} \in \Gamma_{o,\hat{\C}^{\operatorname{IR}}}(\overline{\ttm}_1^{\operatorname{IR}},\overline{\ttm}_2^{\operatorname{IR}})$ implies $(\iota_{1,2}^{\operatorname{IR}})_\sharp \mathfrak{P}^{\operatorname{IR}}\in \RGamma_{o,\C}(\ttM_1,\ttM_2)$. Moreover, $(\phi_1,\phi_1^{\C})$ are optimal Kantorovich potentials for $\calC(\ttM_1,\ttM_2)$ if and only if $(\hat{\phi}_1^{\operatorname{IR}},(\hat{\phi}_1^{\operatorname{IR}})^{\hat{\C}^{\operatorname{IR}}})$ are optimal Kantorovich potentials for $\hat{\calC}^{\operatorname{IR}}(\overline{\ttm}_1^{\operatorname{IR}},\overline{\ttm}_2^{\operatorname{IR}})$.
		\end{enumerate}
	\end{theorem}
	
	\begin{remark}
		The results presented in \ref{subsec: lions cont} and \ref{subsec: lions IR} may be used, following \cite{PS25convex}, to have alternative proofs for Propositions \ref{prop: char tot superdiff} and \ref{prop: char tot mon}, and consequently for Theorem \ref{thm: char opt RGamma} as well.
	\end{remark}
	
	In the Wasserstein example, i.e. $\X_1 = \X_2 = \X$ with $\cost(x_1,x_2) = \rmd^p(x_1,x_2)$ and $p\in[1,+\infty)$, the previous lifting can be rewritten as
	\[\iota\colon \XX_p \to \PP_{p}(\X), \qquad \XX_p:= L^p(\rmQ,\mathcal{F}_{\rmQ},\bbM; \X).\]
	In particular, if $\X$ is a separable Banach space, then $\XX_p$ itself is a Banach space. This structure will play a fundamental role in Section \ref{sec: strict}.
	
	\section{The strict Monge problem between laws of random measures}\label{sec: strict}
	The goal of this section is to introduce and study the strict Monge formulation of optimal transport between (laws of) random measures.
	
	\begin{problem}
		Using the setting of Section \ref{sec: OT gen cost}, let $\ttM_i \in \PP(\PP(\X_i))$, for $i=1,2$ and $\cost\colon\X_1\times \X_2 \to [0,+\infty]$ a lower semicontinuous function. The strict Monge formulation of the optimal transport between $\ttM_1$ and $\ttM_2$ is
		\begin{equation}\label{eq: strict cost}
			\begin{aligned}
				\calC_{sM}(\ttM_1,\ttM_2):=\inf\Bigg\{\int_{\PP(\X_1)} \int_{\X_1} \cost(x_1,\rmT(x_1,\mu_1)) d\mu_1(x_1) d\ttM_1(\mu_1) \, : \quad
				\\
				\rmT\colon\X_1 \times \PP(\X_1) \to \X_2 \text{ Borel,} \ \cT_\sharp \ttM_1 = \ttM_2\Bigg\},
			\end{aligned}
		\end{equation}
		where we denote $\cT(\mu_1) := \rmT(\cdot,\mu_1)_\sharp \mu_1$.
	\end{problem} 
	
	In this scenario, the induced map $\cT\colon\PP(\X_1) \to \PP(\X_2)$ is Borel measurable, applying \cite[Corollary D.6]{Pinzi-Savare25}, and it satisfies
	\[\int_{\PP(\X_1)} \C(\mu_1, \cT(\mu_1)) d\ttM_1(\mu_1) \leq \int_{\PP(\X_1)} \int_{\X_1} \cost(x_1,\rmT(x_1,\mu_1)) d\mu_1(x_1) d\ttM_1(\mu_1),\]
	so that $\cT$ is a competitor for the classic Monge problem
	\begin{equation}\label{eq: monge cost}
		\calC_{M}(\ttM_1,\ttM_2):=\inf\left\{\int_{\PP(\X_1)} \C(\mu_1, \cT(\mu_1)) d\ttM_1(\mu_1)  \,: \, \cT\colon\PP(\X_1) \to \PP(\X_2) \text{ Borel s.t. }\cT_\sharp \ttM_1 = \ttM_2\right\}.
	\end{equation}
	In particular, we have the following inequalities
	\begin{equation}\label{eq: ineq monge}
		\calC(\ttM_1,\ttM_2) \leq \calC_M(\ttM_1,\ttM_2) \leq \calC_{sM}(\ttM_1,\ttM_2).
	\end{equation}

	\noindent The goal of the following subsections will be:
	\begin{enumerate}
		\item to give sufficient conditions under which $\calC(\ttM_1,\ttM_2) = \calC_{sM}(\ttM_1,\ttM_2)$, in the spirit of the result by A. Pratelli \cite{pratelli2007equality};
		\item to show existence and uniqueness of solutions to the strict Monge problem in the case in which $\X_1=\X_2 = \rmB$, where $\big(\rmB,\|\cdot\|\big)$ is a strictly convex and separable Banach space, with $\cost(x_1,x_2) = \|x_1- x_2\|^p$, for $p>1$, as already studied in the Hilbertian setting with $p=2$ in \cite{PS25convex}. We stress that this result is new in the Hilbert case for $p\neq 2$, and for non-Hilbertian norms in the case of $p=2$ (even in finite dimension).
	\end{enumerate}

	\subsection{On the equality between strict Monge and Kantorovich formulation}
	Before proceeding, we introduce the sets of atomless/diffuse and deterministic probability measures that will play an important role in the following. Let $\X_1$ and $\X_2$ be two Polish spaces.
	\begin{equation}\label{eq: diff&det}
		\begin{gathered}
			\PPdiff(\X_1):= \left\{ \mu\in \PP(\X_1) : \mu(\{x_1\}) = 0 \ \forall x_1\in \X_1\right\},
			\\
			\PPdet(\X_1 \times \X_2) := \left\{ \pi\in \PP(\X_1\times \X_2) : \exists f\colon\X_1\to \X_2 \text{ s.t. }\pi = (\operatorname{id}\times f)_\sharp \mu_1, \ \mu_1= \operatorname{pr}^1_\sharp\pi \right\}.
		\end{gathered}
	\end{equation}
	
	\noindent They are Borel subsets of their respective spaces, see Lemma \ref{lemma: meas diff det}. Let us also recall an important characterization of deterministic couplings.

	\begin{lemma}\label{lemma: conc on graph}
		Let $\pi \in \PP(\X_1\times \X_2)$. Then, $\pi \in \PPdet(\X_1\times \X_2)$ if and only if there exists a measurable set $\Gamma \subset \X_1\times \X_2 $ such that $\pi(\Gamma)=1$ and for $\operatorname{pr}^1_\sharp\pi$-a.e. $x_1 \in \X_1$ there exists a unique $x_2\in \X_2$ such that $(x_1,x_2)\in \Gamma$.
	\end{lemma}
	
	\begin{proof}
		See \cite[Theorem 2.3]{ambrosio2021lectures}.
	\end{proof}
	
	Recall that any competitor $\rmT\colon\X_1\times \PP(\X_1)\to \X_2$ for \eqref{eq: strict cost} directly induces a random coupling $\ttP\in \RGamma(\ttM_1,\ttM_2)$ considering
	\begin{equation}\label{eq: repr fully}
		\ttP:= \int_{\PP(\X_1)} \delta_{(\operatorname{id}, \rmT(\cdot,\mu_1))_\sharp\mu_1} d\ttM_1(\mu_1).
	\end{equation}
	Afterwards, random couplings of this form will be called \textit{fully deterministic}.
	
	\begin{lemma}\label{lemma: fully det}
		An element $\ttP\in \PP(\PP(\X_1\times \X_2))$ is fully deterministic if and only if 
		\[\operatorname{pr}^1_\sharp \text{ is }\ttP\text{-essentially injective} \quad \text{and} \quad \ttP\big(\PPdet(\X_1\times \X_2)\big) = 1.\]
	\end{lemma}
	
	\begin{proof}
		The proof is the very same of \cite[Lemma 6.2]{PS25convex}.
	\end{proof}
	
	The condition that \textit{$\operatorname{pr}^1_\sharp$ is $\ttP$-essentially injective} can be equivalently rephrased as: there exists a $\ttP$-measurable map $\mathcal{V}: \PP(\X_1) \to \PP(\X_1 \times \X_2)$ such that $\operatorname{pr}^1_\sharp (\mathcal{V}(\mu_1)) = \mu_1$ for $\ttM_1$-a.e. $\mu_1 \in \PP(\X_1)$ and $\ttP = \mathcal{V}_\sharp \ttM_1$, where $\ttM_1 := \operatorname{pr}^1_{\sharp\sharp} \ttP$.
	
	\smallskip
	
	Using the notation of \eqref{eq: overline RM}, observe that \eqref{eq: repr fully} is equivalent to
	\begin{equation}\label{eq: repr fully overline}
		\overline{\ttP}= \int_{\X_1\times \PP(\X_1)} \delta_{\big((x_1,\rmT(x_1,\mu_1)), (\operatorname{id},\rmT(\cdot,\mu_1))_\sharp\mu_1\big)} d\overline{\ttM}_1(x_1,\mu_1) \in \PP(\X_1 \times \X_2 \times \PP(\X_1\times \X_2)).
	\end{equation}

	\begin{theorem}\label{thm: pratelli}
		For $i=1,2$, let $\ttM_i \in \PP(\PP(\X_i))$. Let $\cost\colon\X_1\times \X_2 \to [0,+\infty)$ be continuous and bounded. Assume $\ttM_1\in \PPdiff(\PP(\X_1))$ and $\ttM_1(\PPdiff(\X_1)) = 1$.
		Then $\calC(\ttM_1,\ttM_2) = \calC_{sM}(\ttM_1,\ttM_2)$.
	\end{theorem}
	
	\begin{proof}
		Since $\cost$ is bounded, \cite[Theorem 6.8]{ambrosio2021lectures} implies the continuity (and boundedness as well) of $\C\colon \PP(\X_1)\times \PP(\X_2)\to [0,+\infty)$. Then, we can apply \cite[Theorem B]{pratelli2007equality}, so that $\calC(\ttM_1,\ttM_2) = \calC_M(\ttM_1,\ttM_2)$, and in particular for every $\varepsilon>0$ there exists $\Pi^\varepsilon \in \Gamma(\ttM_1,\ttM_2) \cap \PPdet(\PP(\X_1)\times \PP(\X_2))\text{ such that }$
		\[\calC_M(\ttM_1,\ttM_2) \leq \int \C(\mu_1,\mu_2) d\Pi^{\varepsilon}(\mu_1,\mu_2) \leq \calC(\ttM_1,\ttM_2) + \frac\varepsilon2.\]
		In particular, for all $\varepsilon>0$ there exists $\cT^\varepsilon\colon\PP(\X_1) \to \PP(\X_2)$ such that $\Pi^\varepsilon = (\operatorname{id},\cT^\varepsilon)_\sharp \ttM_1$.
		
		Now, fix $\varepsilon>0$. The marginal map $(\operatorname{pr}^1_\sharp,\operatorname{pr}^2_\sharp)$ is surjective from the set
		\[\cA_\varepsilon:= \left\{ \pi \in \PPdet(\X_1\times \X_2) : \operatorname{pr}^1_\sharp\pi\in \PPdiff(\X_1), \,  \int_{\X_1\times \X_2} \cost(x_1,x_2) d\pi(x_1,x_2) \leq \C(\operatorname{pr}^1_\sharp\pi, \operatorname{pr}^2_\sharp \pi)+\frac \varepsilon 2\right\}\]
		to the whole product space $\PPdiff(\X_1)\times \PP(\X_2)$. This is again a consequence of \cite[Theorem B]{pratelli2007equality}, and the continuity of $\cost$. Thanks to Lemma \ref{lemma: meas diff det} and the measurability of the other involved operations, the set $\mathcal{A}_\varepsilon$ is Borel measurable as a subset of $\PP(\X_1\times \X_2)$. Then, there exists a universally measurable map $G^\varepsilon\colon\PPdiff(\X_1)\times \PP(\X_2) \to \cA_\varepsilon$ that is a right-inverse for $(\operatorname{pr}^1_\sharp,\operatorname{pr}^2_\sharp)$, see e.g. \cite[Theorem 6.9.1]{Bogachev07}, that we can use to define 
		\[\ttP^\varepsilon:= G^\varepsilon_\sharp \Pi^\varepsilon \in \PP(\PP(\X_1\times \X_2)).\]
		This is well-defined because the first marginal of $\Pi^\varepsilon$ is $\ttM_1$, which is concentrated on $\PPdiff(\X_1)$. Since $G^\varepsilon$ is a right-inverse of the marginal map $(\operatorname{pr}^1_\sharp,\operatorname{pr}^2_\sharp)$, we have that $\Pi^\varepsilon = (\operatorname{pr}^1_\sharp, \operatorname{pr}^2_\sharp)_\sharp \ttP^\varepsilon$. Moreover, $\ttP^\varepsilon$ is concentrated on $\cA_\varepsilon\subset \PPdet(\X_1\times \X_2)$ and $\ttP^{\varepsilon} = \mathcal{V}^\varepsilon_\sharp \ttM_1$, where $\mathcal{V}^\varepsilon(\mu_1) = G^\varepsilon(\mu_1, \cT^\varepsilon(\mu_1))$. Then $\operatorname{pr}^1_\sharp$ is $\ttP$-essentially injective and by Lemma \ref{lemma: fully det} there exists $\rmT^\varepsilon\colon\X_1\times \PP(\X_1) \to \X_2$ representing $\ttP^\varepsilon$ as in \eqref{eq: repr fully}. In particular
		\begin{equation*}
			\begin{aligned}
				\hspace{2cm}\calC_{sM}(\ttM_1,\ttM_2) \leq & \int_{\PP(\X_1)} \int_{\X_1} \cost(x_1,\rmT^\varepsilon(x_1,\mu_1)) d\mu_1(x_1) d\ttM_1(\mu_1) 
				\\
				= & \int_{\PP(\X_1 \times \X_2)} \int_{\X_1\times \X_2} \cost(x_1,x_2) d\pi(x_1,x_2) d\ttP^{\varepsilon}(\pi)
				\\
				= &
				\int_{\PP(\X_1)\times \PP(\X_2)} \int_{\X_1\times \X_2} \cost(x_1,x_2) d\big[G^\varepsilon(\mu_1,\mu_2)\big](x_1,x_2) d\Pi^{\varepsilon}(\mu_1,\mu_2)
				\\
				\leq &
				\int_{\PP(\X_1) \times \PP(\X_2)} \C(\mu_1,\mu_2) d\Pi^{\varepsilon}(\mu_1,\mu_2) + \frac\varepsilon2
				\\
				\leq & \calC(\ttM_1,\ttM_2) + \varepsilon. \hspace{9.75cm} \qedhere
			\end{aligned}
		\end{equation*}
	\end{proof}

	\noindent The same result holds in the setting of Examples \ref{ex: wass} and \ref{ex: wass 2}.
	
	\begin{theorem}\label{thm: pratelli2}
		Let $(\X,\rmd)$ be a complete and separable metric space. Let $p\in[1,+\infty)$ and assume $\cost(x_1,x_2) = \rmd^p(x_1,x_2)$. Let $\ttM_1,\ttM_2 \in \PP_p(\PP_p(\X))$, that is 
		\[\int_{\PP(\X)} \int_{\X} \rmd^p(x_1,x_0) d\mu_1(x_1) d\ttM_1(\mu_1) + \int_{\PP(\X)} \int_{\X} \rmd^p(x_2,x_0) d\mu_2(x_2) d\ttM_2(\mu_2)<+\infty \qquad \text{for some }x_0\in \X. \]
		Assume that $\ttM_1\in \PPdiff(\PP(\X))$ and $\ttM_1(\PPdiff(\X)) =1$. Then for every $\varepsilon>0$ there exists $\rmT^\varepsilon\colon\X \times \PP_p(\X) \to \X$ such that, say $\cT^\varepsilon(\mu_1) := \rmT^\varepsilon(\cdot,\mu_1)_\sharp\mu_1$,
		\begin{equation}
			\cT^\varepsilon_\sharp \ttM_1 = \ttM_2 \qquad \text{and} \qquad \int_{\PP_p(\X)} \int_{\X} \rmd^p(x,\rmT^\varepsilon(x,\mu)) d\mu(x) d\ttM_1(\mu) \leq \sfW_p^p(\ttM_1,\ttM_2) +\varepsilon.
		\end{equation}
	\end{theorem}
	
	\begin{proof}
		The hypothesis imply that $\ttM_1$ and $\ttM_2$ are concentrated on the complete and separable metric space $\PP_p(\X)$, endowed with the Wasserstein metric $\sfw_p$. Moreover, $\sfw_p^p\colon\PP_p(\X) \times \PP_p(\X) \to [0,+\infty)$ is finitely valued and continuous. Then, the same argument of the previous proof allows us to conclude. 
	\end{proof}
	
	\normalcolor
	
	It is interesting to notice that it may happen $\calC(\ttM_1,\ttM_2) = \calC_M(\ttM_1,\ttM_2)<+\infty$, but the set of competitors for the strict Monge \eqref{eq: strict cost} formulation is empty, giving $\calC_{sM}(\ttM_1,\ttM_2) = +\infty$. 
	
	\begin{example}
		Let $\cost:\X_1\times \X_2 \to [0,+\infty)$ be bounded.
		
		(1) Let $\mu_1 \in \PPdiff(\X_1)$. Define $\ttM_1 := \chi_\sharp \mu_1$, where $\chi(x):= \delta_x$. Then, clearly $\ttM_1 \in \PPdiff(\PP(\X_1))$, so that, thanks to the continuity of $\C:\PP(\X_1)\times \PP(\X_2) \to [0,+\infty)$ and \cite[Theorem B]{pratelli2007equality}, for every $\ttM_2\in \PP(\PP(\X_2))$ it holds $\calC(\ttM_1,\ttM_2) = \calC_M(\ttM_1,\ttM_2)<+\infty$. However, as soon as $\ttM_2$ is not concentrated on the image of $\chi$, we have $\calC_{sM}(\ttM_1,\ttM_2) = +\infty$ due to the emptiness of the set of competitors.
		
		(2) Let $ \operatorname{em}_i: \operatorname{Sim}^\infty \times \X_i^\infty \to \PP(\X_i)$ be defined as
		\[ \operatorname{Sim} ^\infty:= \left\{ (a_j)_{j\in \N} \in [0,1]^\infty : a_j\geq 0,\ a_j\geq a_{j+1}, \  \sum_{j\in \N} a_j = 1,\right\}, \quad \operatorname{em}\big( (a_i)_{i\in \N}, (x_i)_{i\in \N} \big) := \sum_{i\in \N} a_i \delta_{x_i}.\]
		For every measure $\sigma \in \PP([0,1]^\infty)$ concentrated on $\operatorname{Sim}^\infty$ and any measure $\Sigma \in \PP(\X_1^\infty)$, define $\ttM_1 = \ttM_{\sigma,\Sigma} := (\operatorname{em}_1)_\sharp (\sigma\otimes \Sigma)$. Then, for every $\ttM_2 \in \PP(\PP(\X_2))$ such that $\ttM_2(\PPdiff(\X_2)) >0$, there are no competitors for \eqref{eq: strict cost}. Also, if $\ttM_2$ is built in the same way but with an incompatible laws for weights $\sigma'\in \PP(\operatorname{Sim}^\infty)$, then the same holds. Indeed, assume $\ttM_2 = (\operatorname{em}_2)_\sharp (\sigma' \otimes \Sigma')$, for some $\Sigma'\in \PP(\X_1^\infty)$ and $\sigma'\in \PP([0,1]^\infty)$ concentrated on $\operatorname{Sim}^\infty$, and assume that there exists $\rmT:\X_1\times \PP(\X_1) \to \X_2$ such that $\cT_\sharp \ttM_1 = \ttM_2$. Then, for $\ttM_1$-a.e. $\mu_1 = \sum_{i\in \N} a_i \delta_{x_i}$, it holds $\cT(\mu_1) = \sum_{i\in \N} a_i \delta_{\rmT(x_i, \mu_1)}$. In particular, the weights can merge if $\rmT(x_i,\mu_1) = \rmT(x_{k},\mu_1)$ for some $i\neq k$, but this may contradict the existence of such a map $\rmT$. For example, if $\sigma$-a.e. $(a_i)$ is such that $a_1\geq 1/2$ and $\sigma'$-a.e. $(a_i')$ satisfies $a_1'<1/2$, then $\rmT$ cannot exist.
	\end{example}
	
	In general, the problem of whether $\calC_M = \calC_{sM}$ relates to the following question: given $\cT\colon\PP(\X_1) \to \PP(\X_2)$, under which assumptions there exists a function $\rmT\colon\X_1\times \PP(\X_1) \to \X_2$ such that $\cT(\mu_1) = \rmT(\cdot,\mu_1)_\sharp\mu_1 $? For a deeper understanding on this problem, we refer to \cite{LS25}.

	\subsection{Existence and uniqueness of the \texorpdfstring{$L^p$}{}-strict Monge problem}\label{subsec: BreBan}
	
	\subsubsection{The \texorpdfstring{$L^p$}{}-Monge problem in Banach spaces}\label{subsub: banach monge}

	Before proceeding, in this subsection we clarify under which set of assumptions we have a unique solution to the $L^p$-Monge problem in a separable Banach space $(\rmB,\|\cdot\|)$. Recall that a measure $\mu\in \PP(\rmB)$ is said to be a \textit{non-degenerate Gaussian}, and we write $\mu\in \PP^{\mathsf g}(\rmB)$, if 
	\begin{equation*}
		\text{for all }z\in \rmB^*\setminus\{0\} \ z_\sharp \mu\in \PP(\R) \text{ is a non-degenerate Gaussian}.
	\end{equation*}
	Afterwards, we say that a Borel subset $N\subset \rmB$ is \textit{Gaussian null} if 
	\begin{equation*}
		\mu(N) = 0 \quad \forall \mu\in \PP^{\sfg}(\rmB).
	\end{equation*}
	Moreover, we will say that a measure is Gaussian regular, i.e. $\mu\in \PP^{\mathsf{gr}}(\rmB)$ if $\mu(N) = 0$ for all Gaussian null set $N\subset \rmB$. Recall that, if $\rmB = \R^d$ (whatever its norm is) the Gaussian null sets coincide with Borel sets that have null Lebesgue measure, and Gaussian regular measures are the ones absolutely continuous with respect to Lebesgue.
	
	The interest in Gaussian null sets is justified by the next theorem, which is a generalization to Banach spaces of the classic Rademacher's theorem. 
	
	\begin{theorem}\label{thm: G-diff Lip}
		Let $\mathcal{O}\subset \rmB$ be a non-empty open set and $f\colon\mathcal{O}\to \R$. If $f$ is locally Lipschitz, then there exists $G\subset \mathcal{O}$ Gaussian null such that $f$ is Gateaux differentiable at every $x\in \mathcal{O}\setminus G$.
	\end{theorem}
	
	Its proof is a consequence of the main results in \cite{Aronszajn76,Phelps78} and follows from this argument: Aronszajn defined a class of null sets for which the points of (Gateaux) non-differentiability of a Lipschitz function form a null set; then, Phelps proved that every Aronszajn null set is Gaussian null, yielding in particular Theorem \ref{thm: G-diff Lip}. It is worth mentioning that Cs\"ornyei \cite{Csonryei99} proved that the class of Aronszajn null sets coincides with the one of Gaussian null sets, which also coincides with the class of cube null sets.

	An important tool for us is the dual multifunction, which encodes the strict convexity of the norm:
	\begin{equation}
		J\colon\rmB\rightrightarrows \rmB^*, \quad J(x) = \left\{z \in \rmB^* \, : \, z(x) = \|x\|^2, \, \|z\|_* = \|x\|\right\}.
	\end{equation}
	Notice that $J(0) = \{0\}$ and $0 \notin J(x)$ if $x\neq 0$.
	
	\begin{lemma}\label{lemma: strict conv norm}
		Let $(\rmB,\|\cdot\|)$ be a Banach space. The following are equivalent:
		\begin{enumerate}
			\item $\|\cdot\|$ is strictly convex, that is 
			\[x,y \in \rmB, \ x\neq y, \ \|x\| = \|y\| = 1 \implies \left\|\frac{x+y}{2}\right\| <1;\]
			\item $J(x) \cap J(y) = \emptyset$ whenever $x\neq y$.
		\end{enumerate}
		In this case, given $p>1$, we also have $J_p(x)\cap J_p(y) = \emptyset$ for $x\neq y$, where $J_p(x):= p\|x\|^{p-2}J(x)$ for $x\neq 0$ and $J_p(0) = \{0\}$.
	\end{lemma}
	
	\begin{proof}
		(1.$\Rightarrow$2.) By contradiction, assume there exists $x\neq y$ and $z \in J(x)\cap J(y)$. This means that $\|x\| = \|z\|_* = \|y\| \neq 0$. Then we have a contradiction
		\[\left\| \frac{x/\|x\|+y/\|y\|}{2} \right\| \geq\frac{1}{\|z\|_*^2} z\left( \frac{x+y}{2} \right) = \frac{1}{\|z\|_*^2}\frac{\|x\|^2+ \|y\|^2}{2}= 1. \]
		
		(2.$\Rightarrow$1.) Assume there exist $x,y \in \rmB$, $x\neq y$, $\|x\| = \| y\| = 1$ and $\|\frac{x+y}{2}\|=1$. By Hahn-Banach theorem there exists $z\in \rmB^*$ such that $\|z\| = 1$ and $z(\frac{x+y}{2}) = 1$. Then the trivial inequality
		\[1 = \frac{z(x) + z(y)}{2} \leq \frac{\|x\| + \|y\|}{2} = 1,\]
		implies $z(x) = z(y) = 1$, and in particular $z\in J(x) \cap J(y)$, yielding a contradiction.
		
		The last statement easily follows: assume $x\neq 0$ and $y\neq 0$, otherwise the result is trivial. If $w \in J_p(x)\cap J_p(y)$, there exist $z\in J(x)$ and $z'\in J(y)$ such that $w = p\|x\|^{p-2} z = p\|y\|^{p-2}z'$. By definition of $J$, it immediately follows that $\|x\| = \|y\|$, and then $ z = z'$, having a contradiction.
	\end{proof}
	
	\begin{remark}\label{rem: gsubdiff}
		Recall that $J_p(x)$ coincides with the Gateaux subdifferential of the convex map $x\mapsto \|x\|^p$ for $p>1$.
	\end{remark}

	\begin{theorem}\label{thm: Monge Banach}
		Let $p>1$ and assume that $(\rmB,\|\cdot\|)$ is a separable Banach space with strictly convex norm. Assume $\mu_1 \in \PP^{\mathsf{gr}}_p(\rmB)$ and $\mu_2 \in \PP_p(\rmB)$. Then, there exists a unique optimal transport plan $\pi\in \Gamma_{o,p}(\mu_1,\mu_2)$ and it is deterministic, that is 
		\begin{equation}
			\text{there exists }\rmt\colon \rmB \to \rmB \text{ measurable such that }\pi = (\operatorname{id},\rmt)_\sharp \mu_1.
		\end{equation}
	\end{theorem}	
	
	\begin{proof}
		Assume first that $\mu_2$ is supported on $B_R := \{\|x\|\leq R\}$ for some $R>0$. Let $\pi\in \Gamma_{o,p}(\mu_1,\mu_2)$ and define $\ttc(x_1,x_2):= \|x_1-x_2\|^p$. Consider an associated optimal $\ttc$-concave potential $\phi\colon\rmB \to [-\infty,+\infty)$ and define 
		\[\varphi(x_1):= \inf_{\|x_2\|\leq R} \|x_1-x_2\|^p - \phi^\ttc(x_2).\]
		We first show that $\varphi$ is locally Lipschitz. Let $x_1,x_1'\in \rmB$ with $\|x_1\|, \|x_1'\|\leq L$. For all $\varepsilon>0$ there exists $x_{2}'\in B_R$ such that $\|x_1'-x_{2}'\|^p - \phi^{\ttc}(x_{2}') \leq  \varphi(x_1') +\varepsilon$ . Then
		\begin{align*}
			\varphi(x_1) - \varphi(x_1') & \leq \varphi(x_1) -  \left(\|x_1'-x_{2}'\|^p - \phi^{\ttc}(x_{2}') - \varepsilon \right)
			\\
			& =
			\|x_1-x_{2}'\|^p - \|x_1'-x_{2}'\|^p +\varepsilon \leq \kappa (p,L,R) \|x_1 - x_1'\| +\varepsilon,
		\end{align*}
		where we used the Lipschitz continuity of $f(r):= |r|^p$ in $[-(L+R),(L+R)]$, with constant $\kappa(p,L,R) = p(L+R)^{p-1}$. By arbitrariness of $\varepsilon$ and switching the role of $x_1$ and $x_1'$, we have the local Lipschitz continuity of $\varphi$. 
		
		Moreover, $\varphi$ is another optimal Kantorovich potential associated to $\pi$, which coincides with $\phi$ $\mu_1$-a.e. Indeed, it is the $\ttc$-transform of $\psi \colon \rmB \to [-\infty,+\infty)$, $\psi(x_2):= \phi^{\ttc}(x_2) - \chi_{B_R}$, where $\chi_{B_R}$ is $0$ in $B_R$ and $+\infty$ outside, and since $\varphi \geq \phi$, we have
		\[\int \ttc d\pi = \int \phi d\mu_1 + \int \phi^{\ttc}d\mu_2 \leq \int \varphi d\mu_1 + \int \psi d\mu_2 \leq \int \ttc d\pi.\]
		Then, $\pi$ is concentrated on $\partial^+_{\ttc}\varphi$. Now, the goal is to show that $\partial_{\ttc}^+\varphi(x_1)$ is a singleton for $\mu_1$-a.e. $x_1 \in \rmB$. By definition of $\ttc$-superdifferential, $x_2\in \partial_{\ttc}^+ \varphi(x_1)$ if and only if $\|\cdot - x_2\|^p - \varphi(\cdot)$ is minimal at $x_1$, which means that $0\in \partial_{G}\big(\|\cdot - x_2\|^p - \varphi(\cdot)  \big)(x_1)$. Thanks to Remark \ref{rem: gsubdiff} and the fact that $\varphi$ is $\mu_1$-a.e. differentiable, we have 
		\begin{equation}\label{eq: gat}
			J_p(x_1-x_2) - D_G\varphi(x_1) \ni 0 \quad \text{ for }\mu_1\text{-a.e. }x_1\in \rmB.
		\end{equation}
		
		Thus, consider the set $E\subset \rmB$ in which $\varphi$ is Gateaux differentiable, which satisfies $\mu_1(\rmB\setminus E) = 0$. If $x_1\in E$ and $x_2,x_2' \in \partial_{\cost}^+ \varphi(x_1)$, then $D_G\varphi(x_1) \in J_p(x_1 - x_2) \cap J_p(x_1 - x_2')$, and Lemma \ref{lemma: strict conv norm} implies $x_2 = x_2'$. 
		
		This shows that $\partial_{\ttc}^+\varphi(x_1)$ is a singleton for $\mu_1$-a.e. $x_1$. Since $\pi$ is concentrated on it, there exists a Borel measurable map $\rmt:\rmB \to \rmB$ such that $\pi = (\operatorname{id},\rmt)_\sharp \mu_1$ (see Lemma \ref{lemma: conc on graph}). We just proved that any optimal coupling $\pi$ is induced by a map. Following a classical disintegration argument, this also implies uniqueness of the optimal coupling.
		
		For the general case in which we do not have a control on the support of $\mu_2$, let $\pi \in \Gamma_{o,p}(\mu_1,\mu_2)$ and define 
		\[\pi_n:= \pi|_{\rmB\times B_n}.\]
		Its first marginal is absolutely continuous with respect to $\mu_1$, and its second marginal has support in $B_n$. Since $\pi_n\leq \pi$, any $\pi_n$ is optimal (up to renormalization) because of Theorem \ref{thm:char opt plans}, and from the previous argument is induced by a map $\rmt_n:\rmB \to B_n$. By uniqueness and the fact that $\pi_n\leq \pi_{n'}$ if $n<n'$, we have $\rmt_n(x_1) = \rmt_{n'}(x_1)$ for $p^1_\sharp \pi_n$-a.e. $x_1\in \rmB$. Then, define the measurable sets
		\[\tilde\Gamma_n:=\big\{ (x_1,x_2)\in \rmB\times \rmB : x_2 = \rmt_n(x_1)\big\}, \quad \Gamma_n:= \bigcup_{n'\geq n} \tilde\Gamma_{n'}, \quad \Gamma = \bigcap_{n\in \N} \Gamma_n.\]
		Thanks to the previous argument, $\pi_k$ is concentrated on $\tilde\Gamma_k$ for all $k\in \N$,
		so that for all $n\in \N$
		\[\pi(\Gamma_n) = \lim_{k\to+\infty} \pi_k(\Gamma_n) = \lim_{k\to+\infty}\pi_k(\tilde\Gamma_k) = \lim_{k\to+\infty}\pi_k(\rmB \times B_k) = \lim_{k\to+\infty}\mu_2(B_k)= 1,\]
		and in particular $\pi(\Gamma)=1$.
		Moreover, for all $k\in \N$ it holds
		\[
		\begin{aligned}
			\operatorname{pr}^1_\sharp \pi_k & \Big(\big\{ x_1 : \exists x_2\neq x_2' \text{ s.t. } (x_1,x_2), (x_1,x_2')\in \Gamma \big\}\Big) 
			\\
			& = \operatorname{pr}^1_\sharp \pi_k\Big(\big\{ x_1: \forall k' \ \exists n,n'\geq k' \text{ s.t. }\rmt_{n}(x_1) \neq \rmt_{n'}(x_1) \big\}\Big)
			\\
			& \leq \operatorname{pr}^1_\sharp \pi_k \Big(\big\{ x_1: \ \exists n,n'\geq k \text{ s.t. }\rmt_{n}(x_1) \neq \rmt_{n'}(x_1) \big\}\Big) = 0.
		\end{aligned}\]
		Since $\operatorname{pr}^1_\sharp \pi_k(A) \nearrow \mu_1(A)$ for all Borel set $A$, we can pass to the limit on the left hand side and obtain that for $\mu_1$-a.e. $x_1\in \rmB$ there exists a unique $x_2$ such that $(x_1,x_2)\in \Gamma$. Then, thanks to Lemma \ref{lemma: conc on graph}, $\pi$ is deterministic, and uniqueness follows by the same disintegration argument.		
	\end{proof}
	
	\begin{remark}\label{rem: infty norm}
		The same result does not hold if we do not assume strict convexity. Let $\rmB = \R^2$ and $\|\cdot\| = \|\cdot\|_\infty$. Let $\mu_1$ and $\mu_2$ be, respectively, the uniform distribution of the rectangle $[0,1]\times [0,2]$ and $[2,3]\times [0,2]$. Clearly, the constant map $(x,y) \mapsto (x+2,y)$ induces an optimal coupling for every $p\in [1,+\infty)$. However, also the map 
		\[(x,y)\mapsto\begin{cases}
			(x+2,y+1) & \text{ if }y\in[0,1)
			\\
			(x+2,y-1) & \text{ if }y\in[1,2]
		\end{cases}\]
		induces an optimal coupling, violating uniqueness.
	\end{remark}

	\subsubsection{Solutions to the strict Monge problem on Wasserstein space}
	First, we specialize the result of Section \ref{subsec: lions cont} to the case of $\PP_p(\rmB)$, where $(\rmB,\|\cdot\|)$ is a separable Banach space and $p\in(1,+\infty)$. In particular, consider $\cost(x_1,x_2):= \|x_1-x_2\|^p$, for which the function $\tta_p(x):= 2^{p-1}\|x\|^p$ is such that \eqref{eq:100} holds with $\tta_1 = \tta_2 = \tta_p$ (see also Examples \ref{ex: wass}, \ref{ex: wass 2} and \ref{ex: wass 3}). Then, the space in \eqref{eq:repr cont}, can be rewritten as the Bochner space
	\begin{equation}
		\cB:= L^p\big( \rmQ, \cF_\rmQ, \bbM; \rmB \big), \quad \|\Z\|_{\cB}^p := \int_{\rmQ} \|\Z(q)\|^p d\bbM(q).
	\end{equation}
	The advantage of this setting is that $\cB$ is itself a separable Banach space.
	As a direct consequence of Lemma \ref{lemma:key}, Claim 4. (see also \cite{CSS2}), we have that the map 
	\begin{equation}
		\iota\colon\big(\cB,\|\cdot\|\big) \to \big(\PP_p(\rmB),\sfw_p\big), \quad \iota(\Z):= \Z_\sharp \bbM,
	\end{equation}
	is surjective, $1$-Lipschitz and for all $\mu_1,\mu_2 \in \PP_p(\rmB)$ and $\Z_1 \in \iota^{-1}(\mu_1)$, there exists a sequence $\{\Z_{2,n}\}\subset \iota^{-1}(\mu_2)$ such that 
	\begin{equation}
		\|\Z_1-\Z_{2,n}\|_{\cB} \to \sfw_p(\mu_1,\mu_2) \quad \text{ as }n\to +\infty.
	\end{equation}
	In particular, a function $\phi\colon\PP_p(\rmB)\to \R$ is $L$-Lipschitz in $\{\mu :\sfw_p(\mu,\delta_0)< R\}$ if and only if the lifted function $\hat{\phi}:= \phi \circ \iota \colon \cB \to \R$ is $L$-Lipschitz in $\{\Z : \|Z\|< R\}$.

	\begin{lemma}\label{lemma: inherit} If $\rmB$ has strictly convex norm, then $\cB$ has strictly convex norm.
	\end{lemma}

	\begin{proof}
		Let \(\Z_1,\Z_2\in \cB\) satisfy
		$\|\Z_1\|_\cB=\|\Z_2\|_\cB=1$ and $\left\|\frac{\Z_1+\Z_2}{2}\right\|_\cB=1.$
		For a.e. \(q\in\rmQ\), using the triangle inequality in \(\rmB\) and the strict convexity of \( r\mapsto r^p\), we have
		\[
		\left(\frac{\|\Z_1(q) + \Z_2(q)\|_\rmB}{2}\right)^p
		\leq \left( \frac{\|\Z_1(q)\|_\rmB + \|\Z_2(q)\|_\rmB}{2} \right)^p \leq 
		\frac{\|\Z_1(q)\|_\rmB^p+ \|\Z_2(q)\|_\rmB^p}{2}.
		\]
		Integrating with respect to $\bbM$, we obtain
		\[
		1 = \left\|\frac{\Z_1+\Z_2}{2}\right\|_\cB^p
		\leq
		\frac{\|\Z_1\|_\cB^p+\|\Z_2\|_\cB^p}{2}
		=1.
		\]
		Therefore equality holds throughout and, in particular, for $\bbM$-a.e. $q\in \rmQ$
		\[
		\left(\frac{\|\Z_1(q) + \Z_2(q)\|_\rmB}{2}\right)^p
		= \left( \frac{\|\Z_1(q)\|_\rmB + \|\Z_2(q)\|_\rmB}{2} \right)^p = 
		\frac{\|\Z_1(q)\|_\rmB^p+\|\Z_2(q)\|_\rmB^p}{2}.
		\]
		By strict convexity of \(t\mapsto t^p\), the equality in the triangle inequality and the strict convexity of $\|\cdot\|_\rmB$, we have $\Z_1(q) = \Z_2(q)$ for $\bbM$-a.e. $q\in \rmQ$, and in particular $\Z_1 = \Z_2$.
	\end{proof}

	\newcommand{\PPtr}{\PP^{\mathsf{tr}}_p}
	
	Let us now define the set of $p$-transport regular measures $\PPtr(\rmB) \subset \PP_p(\rmB)$ as
	\begin{equation}
		\PPtr(\rmB):= \left\{ \mu_1 \in \PP_p(\rmB) : \forall \mu_2 \in \PP_p(\rmB) \ \ \Gamma_{o,p}(\mu_1,\mu_2) \subset \PPdet(\rmB \times \rmB)\right\}.
	\end{equation}
	By a usual disintegration argument, already used in the proof of Theorem \ref{thm: Monge Banach}, such a set coincide with the set of measures $\mu_1 \in \PP_p(\rmB)$ such that for all $\mu_2 \in \PP_p(\rmB)$ there exists a unique optimal coupling, which is also deterministic. In particular, Theorem \ref{thm: Monge Banach} implies that $\PP_p^{\mathsf{gr}}(\rmB) \subseteq \PPtr(\rmB)$.
	\normalcolor

	Next, we exploit the space $\cB$ to transfer properties that are specific of Banach spaces to $\PP_p(\rmB)$. We use this paradigm to set the notions of Gaussian distributions and Gaussian null sets in $\PP_p(\rmB)$.

	\begin{definition}\label{def: gauss reg superreg}
		We say that:
		\begin{enumerate}
			\item $\ttM \in \PP_p(\PP_p(\rmB))$ is a non-degenerate Gaussian measure if there exists $\frg \in \PP_p(\cB)$ non-degenerate Gaussian such that $\ttM = \iota_\sharp \frg$. We will write $\ttM\in \PP^{\sfg}_p(\PP_p(\rmB))$;
			\item $\rmN\subset \PP_p(\rmB)$ is Gaussian null if $\iota^{-1}(\rmN)$ is Gaussian null in $\cB$. Equivalently, we could ask that $\ttM(\rmN) = 0$ for all $\ttM\in \PP_p^\sfg(\PP_p(\rmB))$;
			\item $\ttM\in \PP_p(\PP_p(\rmB))$ is Gaussian regular if $\ttM(\rmN) = 0$ for all Gaussian null sets $\rmN\subset \PP_p(\rmB)$. We will write $\ttM\in \PP^{\mathsf{gr}}_p(\PP_p(\rmB))$;
			\item $\ttM\in \PP_p(\PP_p(\rmB))$ is super Gaussian regular if $\ttM\in \PP^{\mathsf{gr}}(\PP_p(\rmB))$ and is concentrated on $\PP_p^{\mathsf{tr}}(\rmB)$.
			We will write $\ttM\in \PP^{\mathsf{sgr}}_p (\PP_p(\rmB))$.
		\end{enumerate} 
	\end{definition}
	
	A priori, the definition of Gaussian measures defined on $\PP_p(\rmB)$ (and of Gaussian null sets) may depend on the choice of the triplet $(\rmQ,\cF_\rmQ, \bbM)$ used to define $\cB$. However, since all atomless standard Borel spaces are isomorphic, arguing as in \cite[Section 6]{PS25convex}, it is not hard to prove that there is no dependence on such a choice.

	\begin{theorem}\label{thm: BreRM}
		Let $p\in(1,+\infty)$ and $(\rmB,\|\cdot\|)$ a strictly convex and separable Banach space. Let $\ttM_1,\ttM_2 \in \PP_p(\PP_p(\rmB))$ be such that $\ttM_1$ is super Gaussian regular. Then, there exists a unique optimal random coupling $\ttP \in \RGamma_{o,p}(\ttM_1,\ttM_2)$ and it is fully deterministic, i.e. it has the form \eqref{eq: repr fully} for some Borel measurable $\rmT\colon\rmB\times \PP_p(\rmB) \to \rmB$. In particular
		\begin{equation*}
			\sfW_p^p(\ttM_1,\ttM_2) = \int_{\PP_p(\rmB)} \int_{\rmB} \|x_1 - \rmT(x_1,\mu_1)\|_{}^p d\mu_1(x_1)d\ttM_1(\mu_1), \quad \cT_\sharp \ttM_1 = \ttM_2, \quad \cT(\mu_1) = \rmT(\cdot,\mu_1)_\sharp\mu_1. 
		\end{equation*}
	\end{theorem}
	
	\begin{proof}
		\textbf{Step 1}: assume first that $\ttM_2$ is supported on a bounded set of $\PP_p(\rmB)$, say $\{\mu_2 \in \PP_p(\rmB) : \int_{\rmB} \|x_2\|^pd\mu_2(x_2) \leq R^p \}$ for some $R>0$. Let $\ttP \in \RGamma_{o,p}(\ttM_1,\ttM_2)$. Thanks to Lemma \ref{lemma: rand coupl}, $\ttP$ is concentrated on optimal couplings, that is for $\ttP$-a.e. $\pi \in \PP_p(\rmB\times \rmB)$ it holds $\pi \in \Gamma_{o,p}( \mu_1^\pi,\mu_2^\pi)$, where $\mu_i^\pi :=\operatorname{pr}^i_\sharp\pi$. Moreover, since $\operatorname{pr}^1_{\sharp\sharp}\ttP = \ttM_1$ is concentrated on $\PP^{\mathsf{tr}}_p(\rmB)$, $\ttP$-a.e. $\pi$ is in $\PPdet(\X_1\times \X_2)$; in particular there exists $\rmt^\pi \in L^p(\rmB, \mu_1^\pi;\rmB)$ such that $\pi = (\operatorname{id},t^\pi)_\sharp \mu_1^\pi$ and $\mu_2^\pi = \rmt^\pi_\sharp \mu_1^\pi$. 
		
		\textbf{Step 2}: we show that $\Pi:= \big(\operatorname{pr}^1_\sharp, \operatorname{pr}^2_\sharp \big)_\sharp \ttP \in \PPdet\big(\PP_p(\rmB)\times \PP_p(\rmB)\big)$.
		In particular, there exists $\cT\colon\PP_p(\rmB) \to \PP_p(\rmB)$ such that $\Pi = (\operatorname{id},\cT)_\sharp \ttM_1$. To this aim, we will rely on Proposition \ref{prop: lift cont}. 
		\\
		Thanks to Lemma \ref{lemma: rand coupl}, $\Pi\in \Gamma_{o,p}(\ttM_1,\ttM_2)$. Consider $\phi \colon \PP_p(\rmB) \to [-\infty,+\infty)$ an associated $\C$-concave optimal Kantorovich potential, where $\C = \sfw_p^p$. Arguing as in the proof of Theorem \ref{thm: Monge Banach}, up to substituting $\phi$ with $\varphi(\mu_1):=\inf\left\{ \sfw_p^p(\mu_1,\mu_2) - \phi^{\C}(\mu_2) : \sfw_p(\mu_2,\delta_0)\leq R \right\}$, we can assume that $\phi$ is locally Lipschitz. From the standard theory, see Theorem \ref{thm:char opt plans}, $\Pi$ is concentrated on $\partial^+_{\C}\phi$. We are done if we prove that for $\ttM_1$-a.e. $\mu_1 \in \PPtr(\rmB)$ the set $\partial^+_{\C}\phi$ is a singleton. Clearly, the latter is non empty for $\ttM_1$-a.e. $\mu_1$. So, assume there exist $\mu_2,\mu_2' \in \partial^+_{\C} \phi(\mu_1)$ and $\mu_1 \in \PP^{\mathsf{tr}}_p(\rmB)$. Then, 
		\begin{itemize}
			\item from Proposition \ref{prop: char tot superdiff}, if $\pi \in \Gamma_{o,p}(\mu_1,\mu_2)$ and $\pi'\in\Gamma_{o,p}(\mu_1,\mu_2')$, we have $\pi,\pi' \in \partial_{\ttt,\C}^+\phi(\mu_1)$;
			\item since $\mu_1 \in \PPtr(\rmB)$, there exists $\rmt^{\pi},\rmt^{\pi'}\in L^p(\rmB,\mu_1;\rmB)$ inducing, respectively, $\pi$ and $\pi'$. Then, thanks to Proposition \ref{prop: lift cont}, for all $Z_1\in \iota^{-1}(\mu_1)$ it holds $\rmt^\pi(\Z_1), \rmt^{\pi'}(\Z_1) \in \partial_{\hat{\C}}^+\hat{\phi}(\Z_1)$. Notice that, in this setting, $\hat{\C}(\Z_1,\Z_2) = \|\Z_1 - \Z_2 \|_{\cB}^p$.
		\end{itemize}
		Since $\mu_2\neq \mu_2'$, this yields $\partial_{\hat{\C}}^+\hat{\phi}(\Z_1)$ is not a singleton for every $Z_1 \in \iota^{-1}(\mu_1)$, and in particular any point in $\iota^{-1}(\mu_1)$ is a point of non-Gateaux differentiability of $\hat{\phi}$, since in the proof of Theorem \ref{thm: Monge Banach} we have already observed that the superdifferential is a singleton at the points where the potential is Gateaux differentiable: this is because $\cB$ inherits separability and strict convexity from $\rmB$ thanks to Lemma \ref{lemma: inherit}. Now, the local Lipschitz continuity of $\phi$ implies the local Lipschitz continuity of $\hat{\phi}$, and this means that there exists a Gaussian null set $N\subset \cB$ such that $\hat{\phi}$ fails to be Gateaux differentiable only on $N$. In particular, $\iota^{-1}(\{\mu_1 \in \PP^{\mathsf{tr}}_p(\rmB) : |\partial_{\C}^+\phi(\mu_1)|>1\}) \subset N$, so that by Gaussian regularity of $\ttM_1$, we have that for $\ttM_1$-a.e. $\mu_1$, the set $\partial_{\C}^+\phi(\mu_1)$ is a singleton.
		
		
		\textbf{Step 3}: denote with $\widetilde{\mathcal{A}} \subset \PP_p(\rmB\times \rmB)$ the set of couplings
		\[\widetilde{\mathcal{A}} := \left\{ \pi \in \PP_p(\rmB\times \rmB) : \operatorname{pr}^1_\sharp \pi \in \PPtr(\rmB), \ \pi \in \Gamma_{o,p}(\operatorname{pr}^1_\sharp \pi , \operatorname{pr}^2_\sharp \pi) \cap \PPdet(\rmB \times \rmB)\right\},\]
		which is $\ttP$-measurable thanks to Lemmas \ref{lemma: meas diff det} and \ref{lemma: meas tr}. From Step 1, we know that $\ttP(\widetilde{\mathcal{A}}) = 1$. Moreover, by inner regularity of $\ttP$ there exists $\mathcal{A}\subset \widetilde{\mathcal{A}}$ that is Borel measurable and $\ttP(\mathcal{A}) = 1$. Moreover, the continuous map $(\operatorname{pr}^1_\sharp, \operatorname{pr}^2_\sharp)$ is bijective from $\widetilde{\mathcal{A}}$ to $\PPtr(\rmB) \times \PP_p(\rmB)$, and thus its inverse, that we denote $G$, is Borel measurable from $(\operatorname{pr}^1_\sharp,\operatorname{pr}^2_\sharp)(\mathcal{A})$ to $\mathcal{A}$, for example thanks to \cite[Lemma 6.7.1]{Bogachev07}. Thus, we have that $\ttP = G_\sharp \Pi$, and from Step 2 we then have that $\ttP = \mathcal{V}_\sharp \ttM_1$, where $\mathcal{V}(\mu_1) := G(\mu_1,\cT(\mu_1))$ for all $\mu_1 \in \PPtr(\rmB)$.

		\normalcolor
		
		\textbf{Step 4}: thanks to Lemma \ref{lemma: fully det}, we proved that for all $\ttP, \ttP' \in \RGamma_{o,p}(\ttM_1,\ttM_2)$ there exist $\rmT, \rmT': \rmB\times \PP_p(\rmB)\to \rmB$, defined $\overline{\ttM}_1$-a.e., representing, respectively, $\ttP$ and $\ttP'$ as in \eqref{eq: repr fully}. Then 
		\[\ttP'':= \frac12\ttP + \frac12\ttP'\] 
		is still optimal between $\ttM_1$ and $\ttM_2$, so that there exists $\rmT''\in L^p(\rmB \times \PP_p(\rmB),\overline{\ttM}_1;\rmB)$ representing it. Then, \eqref{eq: repr fully overline} can be written as 
		\[
		\begin{aligned}
			\ttP''= & \int_{\X_1\times \PP(\X_1)} \delta_{\big((x_1,\rmT''(x_1,\mu_1)), (\operatorname{id},\rmT''(\cdot,\mu_1))_\sharp\mu_1\big)} d\overline{\ttM}_1(x_1,\mu_1)
			\\
			= & \int_{\X_1\times \PP(\X_1)} \frac12\delta_{\big((x_1,\rmT(x_1,\mu_1)), (\operatorname{id},\rmT(\cdot,\mu_1))_\sharp\mu_1\big)} + \frac12\delta_{\big((x_1,\rmT{'}(x_1,\mu_1)), (\operatorname{id},\rmT'(\cdot,\mu_1))_\sharp\mu_1\big)}d\overline{\ttM}_1(x_1,\mu_1).
		\end{aligned}
		\]
		The uniqueness of disintegration with respect to the map $(\operatorname{pr}^1,\operatorname{pr}^1_\sharp)\colon \rmB\times \rmB \times \PP(\rmB \times \rmB) \to \rmB \times \PP(\rmB)$ yields that $\rmT = \rmT'= \rmT''$ $\overline{\ttM}_1$-a.e., and in particular $\ttP=\ttP'$.
		
		\textbf{Step 5}: if we do not have the additional hypothesis on $\ttM_2$, let $\ttP\in \RGamma_{o,p}(\ttM_1,\ttM_2)$ and define 
		\[\ttP_n := \frac{1}{\ttP(A_n)} \ttP|_{A_n}, \ \ A_n:= \left\{\pi : \int \|x_2\|^p d\pi(x_1,x_2)\leq n^p \right\}, \ \ \Pi_n:=\big(\operatorname{pr}^1_\sharp,\operatorname{pr}^2_\sharp\big)_\sharp \ttP_n, \ \ \ttM_{n,i}:= \operatorname{pr}^i_{\sharp\sharp} \ttP_{n}.\]
		Now, $\ttP_n \in \RGamma_{o,p}(\ttM_{n,1},\ttM_{n,2})$ and $\ttM_{n,2}$ has bounded support, so that, from the previous steps and Lemma \ref{lemma: fully det}, there exists $\rmT_n:\rmB\times \PP_p(\rmB)\to \rmB$ representing it. In particular:
		\begin{enumerate}
			\item $\ttP$ is concentrated on deterministic couplings \[\ttP(\PPdet(\X_1\times \X_2)) = \lim_{n\to+\infty} \ttP(\PPdet( \X_1\times \X_2)\cap A_n) = \lim_{n\to+\infty}\ttP_n(\PPdet(\X_1\times \X_2)) = 1; \]
			\item $\Pi_n$ is deterministic, and arguing as in the proof of Theorem \ref{thm: Monge Banach}, we obtain that $\Pi:= \big(\operatorname{pr}^1_\sharp,\operatorname{pr}^2_\sharp\big)_\sharp \ttP$ is deterministic, i.e. there exists $\cT:\PP_p(\rmB) \to \PP_p(\rmB)$ Borel measurable such that $\Pi = (\operatorname{id},\cT)_\sharp \ttM_1$;
			\item arguing as in Step 3, we have also that $\operatorname{pr}^1_\sharp$ is $\ttP$-essentially injective.
			\normalcolor
		\end{enumerate}
		
		This proves that $\ttP$ is fully deterministic, thanks again to Lemma \ref{lemma: fully det}, and by a disintegration argument as in Step 4 we also achieve uniqueness.
	\end{proof} 
	
	\begin{remark}\label{rem: R}
		If $\rmB = \R$, one could also require that $\ttM_1 \in \PP_p^{\mathsf{gr}}(\PP_p(\rmB))$ and is concentrated on $\PPdiff_p(\rmB)$, since the latter is enough to have that the unique optimal coupling between $(\mu_1,\mu_2)$ is induced by a map for $\ttM_1$-a.e. $\mu_1$ and any $\mu_2\in \PP_p(\rmB)$.
	\end{remark}
	
	\begin{proposition}
		If there exists $\ttM\in \PP^{\mathsf{sgr}}_p(\PP_p(\rmB))$ with full support, then $\PP^{\mathsf{sgr}}_p(\PP_p(\rmB))$ is dense in $\PP_p(\PP_p(\rmB))$. In particular, if $\rmB$ is finite dimensional, the previous conclusion holds.
	\end{proposition}
	
	\begin{proof}
		Directly from the definition, it is immediate to verify that $\ttM'\ll \ttM$ implies $\ttM'\in \PP_p^{\mathsf{sgr}}(\PP_p(\rmB))$, so that it is enough to show $\{\ttM' \in \PP_p(\PP_p(\rmB)) : \ttM'\ll\ttM\}$ is dense in $\PP_p(\PP_p(\rmB))$. Since the collection of discrete random measures \[\left\{ \sum_{i=1}^n a_i \delta_{\mu_i} : n\geq 1, a_i>0, \sum_{i=1}^n a_i = 1, \mu_i\in \PP_p(\rmB)\right\}\] 
		is dense, is enough to show that any random measure of such a form can be approximated. So, let $\ttN = \sum_{i=1}^n a_i \delta_{\mu_i}$ with $\mu_i\neq \mu_{i'}$ if $i\neq i'$, and let $r>0$ be such that the $\sfw_p$-balls $\{B(\mu_i,r): i\leq n\}$ are disjoint. Define 
		\[\ttM_r:= \left( \sum_{i=1}^n \frac{a_i}{\ttM(B(\mu_i,r))} \mathds{1}_{B(\mu_i,r)}\right) \ttM \ll \ttM,\]
		which clearly converges to $\ttN$ as $r\to0$.
		
		If $\rmB$ is finite dimensional, say $\rmB = \R^d$ endowed with a norm $\|\cdot\|$, \cite[Theorem 6.25]{PS25convex} shows that taking any non-degenerate Gaussian $\frg \in \PP\big( C^1([0,1]^d,\R^d) \big)$, interpreting $C^1([0,1]^d,\R^d)$ as a subset of $\rmB = L^p\big([0,1]^d, \cF_{[0,1]^d}, \lambda^d; \R^d \big)$, then $\frg\in \PP^{\sfg}_p(\cB)$ and the density of $C^1([0,1]^d,\R^d)$ in $\cB$ gives that $\ttG:=\iota_\sharp \frg \in \PP_p(\PP_p(\R^d))$ is Gaussian and with full support. Relying again on \cite[Theorem 6.25]{PS25convex}, $\ttG$-a.e. $\mu$ is absolutely continuous with respect to the Lebesgue measure, and in finite dimension this is equivalent to be Gaussian regular, implying in particular that $\ttG\in \PP^{\mathsf{sgr}_p}(\PP_p(\rmB))$.
	\end{proof}

	\normalcolor

	\appendix

	\section{Appendix}
	In this appendix we collect some measurability results for subsets of probability measures used along the paper.
	
	\begin{lemma}\label{lemma: meas diff det}
		Let $\X_1$ and $\X_2$ be two Polish spaces. The sets $\PPdiff(\X_1)$ and $\PPdet(\X_1\times \X_2)$ defined in \eqref{eq: diff&det} are Borel measurable in their respective spaces.
	\end{lemma}
	
	\begin{proof}
		The second claim has been proved in \cite{LS25}, so we focus on the first one. Recall that $\mu\in \PPdiff$ if and only if $\mu\otimes \mu (\Delta) = 0$, where $\Delta := \{(x_1,x_1)\in \X_1\times \X_1 : x_1 \in \X_1 \}$. Now, the map $\operatorname{Prod}\colon\PP(\X_1) \to \PP(\X_1\times \X_1)$ mapping $\mu$ to $\mu\otimes \mu$ is continuous. Then, $\PPdiff(\X_1) = \operatorname{Prod}^{-1}(\{\sigma \in \PP(\X_1\times \X_1) : \sigma(\Delta) = 0\})$, which is measurable thanks to \cite[Lemma D.1]{Pinzi-Savare25}.
	\end{proof}
	
	
	\begin{lemma}
		Let $\rmB$ be a separable Banach space. Then:
		\begin{enumerate}
			\item $\PP^{\mathsf g}(\rmB)$ is Borel measurable;
			\item if $\rmB$ is finite dimensional, then $\PP^{\mathsf{gr}}(\rmB)$ is Borel measurable.
		\end{enumerate}
	\end{lemma}

	\begin{proof}
		Regarding the first claim, notice that $\mu \in \PP^{\mathsf{g}}(\rmB)$ if and only if $z_\sharp \mu$ is Gaussian in $\R$ for all $z\in \rmB^*$ and $\mu$ has full support. These two properties can be checked by countably many Borel operations: indeed, it suffices to check that $(z_k)_\sharp \mu$ is Gaussian for all $k\in \N$, where $\{z_k:k\in \N\}\subset \rmB^*$ is weakly* dense. On the other hand, take a countable basis of open sets $\{U_j : j\in \N\}$ for the topology of $\rmB$, then $\mu$ has full support if and only if $\mu(U_j)>0$ for all $j\in \N$. Finally
		\[\PP^{\mathsf g}(\rmB) = \bigcap_{j,k \in \N} \big\{\mu\in \PP(\rmB) : (z_{k})_\sharp \mu \text{ is Gaussian in }\R, \ \mu(U_j) >0 \big\}, \]
		and we conclude since (possibly degenerate) Gaussian measures in $\R$ are closed by narrow convergence, and the maps $\mu \mapsto (z_k)_\sharp \mu$ and $\mu \mapsto \mu(U_j)$ are Borel measurable.
		
		The measurability of Gaussian regular measures in finite dimension follows adapting the same argument of \cite[Proposition 2.6]{PS25convex}.
	\end{proof}

	\begin{lemma}\label{lemma: meas tr}
		The set $\PP^{\mathsf{tr}}_p(\rmB) \subset \PP_p(\rmB)$ is co-Souslin, i.e. $\PP_p(\rmB) \setminus \PP^{\mathsf{tr}}_p(\rmB)$ is Souslin. In particular, it is universally measurable. 
		Moreover, if $\rmB$ is finite dimensional, then $\PP_p^{\mathsf{tr}}(\rmB)$ is also a Borel set.
	\end{lemma}
	
	\begin{proof}
		Thanks to \cite[Lemmas 4.2 \& 4.3]{beiglbock2025brenier}, the set $\mathcal{M}:= \{(\mu_1,\mu_2) \in \PP_p(\rmB) \times \PP_p(\rmB) : \Gamma_{o,p}(\mu_1,\mu_2) \subset \PPdet(\X_1\times \X_2)\}$ is Borel measurable. Note that $\mu_1 \in \PP_p(\rmB) \setminus \PP_p^{\mathsf{tr}}(\rmB)$ if and only if there exists $\mu_2 \in \PP_p(\rmB)$ such that $(\mu_1,\mu_2) \in \mathcal{M}^c,$ so that $\PP_p(\rmB) \setminus \PP_p^{\mathsf{tr}}(\rmB) = \operatorname{pr}^1(\mathcal{M}^c)$ concluding the proof. 
		The last part of the statement is a consequence of \cite[Theorem 4.5]{beiglbock2025brenier}.
	\end{proof}
	
	\printbibliography
	
\end{document}